\documentclass[10pt]{amsart}
\usepackage{color}
\usepackage{enumerate}
\usepackage{amssymb}
\usepackage{multirow}
\usepackage{graphicx}

\newtheorem{theorem}{Theorem}[section]
\newtheorem{lemma}[theorem]{Lemma}
\newtheorem{proposition}[theorem]{Proposition}
\newtheorem{corollary}[theorem]{Corollary}

\theoremstyle{definition}
\newtheorem{definition}[theorem]{Definition}
\newtheorem{example}[theorem]{Example}
\newtheorem{remark}[theorem]{Remark}

\newtheorem{assumption}[theorem]{Assumption}

\newcommand{\R}{\mathbb R}

\newcommand{\Z}{\mathbb Z}
\newcommand{\C}{\mathbb C}

\newcommand{\So}{\mathcal{R}}
\newcommand{\Sm}{\mathcal{RM}}

\newcommand{\SF}{\mathbb S}

\DeclareMathOperator{\pv}{\wedge \mspace{-9.5mu}_* \ }

\newcommand{\N}{\mathbb N}

\newcommand{\HH}{\mathbb H}

\theoremstyle{remark} 

\numberwithin{equation}{section}

\begin{document}

\title[$*$-Log for Slice Regular Functions]{$*$-Logarithm for Slice Regular Functions}

\author[A. Altavilla]{A. Altavilla}\address{Altavilla Amedeo: Dipartimento di Matematica, Universit\`a degli Studi di Bari ``Aldo Moro'', Via Edoardo Orabona, 4, 70125,
Bari, Italia} \email{amedeo.altavilla@uniba.it}

\author[C. de Fabritiis]{C. de Fabritiis}\address{Chiara de Fabritiis: Dipartimento di Ingegneria Industriale e Scienze
Matematiche, Universit\`a Politecnica delle Marche, Via Brecce Bianche, 60131,
Ancona, Italia} \email{fabritiis@dipmat.univpm.it}
\thanks{Partially supported by GNSAGA of INdAM and the INdAM project ``Teoria delle funzioni ipercomplesse e applicazioni''. The authors acknowledge with pleasure the support of FBK-Cirm, Trento, and of the Clifford Research Group at Ghent University where parts of this paper were written.
The first author was also financially supported by a INdAM fellowship 
``{\em Mensilit\`a di borse di studio per l'estero a.a. 2018--2019}'' spent at Ghent University.}

\date{\today}

\subjclass[2010]{Primary 30G35; secondary 30C25, 32A30, 33B10}
\keywords{Slice-regular functions, quaternionic exponential, $*$-exponential,  $*$-logarithm.}

\begin{abstract} 
In this paper, we study the (possible) solutions of the equation 
$\exp_{*}(f)=g$, where $g$ is a slice regular never vanishing function on 
a circular domain of the quaternions $\HH$ and $\exp_{*}$ is the natural generalization
of the usual exponential to the algebra of slice regular functions.
Any function $f$ which satisfies $\exp_{*}(f)=g$ is called a $*$-logarithm of $g$.
We provide necessary and sufficient conditions, expressed in terms of the zero set
of the ``vector'' part $g_{v}$ of $g$, for the existence of a $*$-logarithm of $g$, under a natural topological condition on the domain $\Omega$. By the way, we prove an existence result if $g_{v}$ has no non-real isolated zeroes;
we are also able to give a comprehensive approach to deal with more general cases. 
We are thus able to obtain an existence result when the non-real isolated zeroes of $g_{v}$ are finite, the domain is either the unit ball, or $\HH$, or $\mathbb{D}$ and a further condition on the ``real part'' $g_{0}$ of $g$ is satisfied (see Theorem~\ref{final-result} for a precise statement).
We also find some unexpected uniqueness results, again related to the zero set of $g_{v}$, in sharp contrast with the complex case.
A number of examples are given throughout the paper in order to show 
the sharpness of the required conditions.
\end{abstract}
\maketitle

\section{Introduction}

The aim of this paper is to investigate on (possible) existence and/or uniqueness of solutions of 
\begin{equation}\label{exp}
\exp_{*}(f)=g,
\end{equation}
given a never-vanishing function $g$ which is slice-regular on a circular domain $\Omega$ of the quaternions $\mathbb H$. A solution of this equation will be
called $*$-logarithm of $g$.

The $*$-exponential operator $\exp_*$ on the space $\So(\Omega)$ of slice regular functions was introduced by Colombo, Sabadini and Struppa in~\cite{C-S-St-2} and was later studied by Altavilla and de Fabritiis in~\cite{A-dF}. We underline the fact that its definition in the form 
$\exp_{*}(f)=\sum\frac{f^{*n}}{n!}$, where $f^{*n}=f*\dots*f$ denotes the $*$-product of $f$ with itself $n$ times (i.e., the $n$th $*$-power),
was due to the remark that in general the pointwise product (and thus the pointwise powers) of slice regular functions is not slice regular.
One of the features of $\exp_*(f)$ is that, as anyone would expect, it is a never-vanishing slice regular function for any slice regular $f$, so Equation~(\ref{exp}) has no solution unless $g$ is never-vanishing. 
By an accurate investigation on the features of this operator we will be able 
to understand under which conditions on $\Omega$ and $g$, Equation~\eqref{exp}
admits a solution.

While preparing the final draft of this paper we heard that similar results, suggested by different motivations and obtained by distinct techniques, were
about to be published by Gentili, Prezelij and Vlacci (see~\cite{GPV}).

Slice-regular functions on quaternions were introduced by Gentili and Struppa in 2006 in order to give a suitable notion of regularity for functions of a quaternionic variable which would provide a good balance between two requirements: the first one is 
the necessity of a smooth behaviour, in the sense of existence of some kind of derivatives,   
while the second is the condition that the set of these functions is large enough to offer an interesting theory; for a detailed account on the path which led to this approach, see~\cite{G-S-St}. 
In particular, in our opinion one of the key-points of the theory is the definition of $*$-product which, together with the pointwise sum, gives to the space $\So(\Omega)$ the structure of an associative $*$-algebra. 

In this last 15 years the theory quickly developed in several directions, creating many connections with differential geometry, algebraic geometry, functional analysis, operator theory and applications to physics and engineering.

Before giving a description of the content of our paper, it is worth while looking a bit more thoroughly to the behaviour of the operator $\exp_*$ with the aim to explain the reason of its definition and the analogies and dissimilarities with the exponential in the classical (i.e., complex analytic) sense. 

Indeed, one can define an exponential map $\exp:\mathbb H\to \mathbb H\setminus\{0\}$ on the quaternions which turns out to be slice regular and slice preserving. Nonetheless, in general the composition of slice regular functions is not slice regular (for a more detailed treatment of the composition in the setting of slice regular functions see~\cite{Gal-OGC-Sabadini,R-W,V}). Thus, considering $\exp\circ f$ for any slice regular function $f$ provides a never vanishing function which could be non-regular. Indeed, the equality $\exp\circ f=\exp_*(f)$ holds true when $f$ is a slice preserving function and, on a suitable subset of $\Omega$, for one-slice preserving functions, but in general it does not for any slice regular function (see Definition~\ref{slice-preserving-definition}, Remark~\ref{caso-slice-pres} and Example~\ref{esempi-exp-composizione} for a detailed comparaison between $\exp\circ f$ and $\exp_*(f)$).

Thus, the topological approach used in the theory of holomorphic functions on complex numbers, where the logarithm of $g:\Omega\to \mathbb C\setminus\{0\}$ can be obtained by means of a lifting of the map $g$ with respect to the covering $\exp:\mathbb C\to \mathbb C\setminus\{0\}$, provided $\Omega$ is contractible, cannot be used on slice regular functions unless the function $g$ has special properties (i.e., it preserves at least one slice).
In antithesis with $\exp\circ f$, the definition of $\exp_{*}$ is the natural one, 
since it relies on the $*$-product, it provides an operator which maps $\So(\Omega)$ in $\So(\Omega)$
and it coincides with $\exp\circ f$ on slice preserving functions.

We now give an outline of the paper.

Section~\ref{preliminaries} contains definitions and preliminary material: we recall the basic definitions of the theory of slice regular functions together with the main topological definition we need, that is the notion of slice contractible domain, and we cite the main properties of the $*$-exponential proven in ~\cite{A-dF}.  
Following the approach originally due to Colombo, Gonzales-Cervantes and Sabadini, 
we can write
any slice regular function $f$ as a sum 
$f=f_{0}+f_{1}i+f_{2}j+f_{3}k$, where $i,j,k$ is the standard basis of imaginary quaternions and $f_{\ell}$, for $\ell=0,1,2,3$, are quaternionic valued, slice preserving regular functions.
The function $f_{0}$ can be interpreted as the ``real part'' of $f$ and $f_{v}=f_{1}i+f_{2}j+f_{3}k$ as the ``vector part'' of $f$. Thanks to the splitting $f=f_{0}+f_{v}$  we can give a useful rewriting of $\exp_*(f)$ in terms of the exponential of the real part of $f$ and of two auxiliary slice preserving functions, namely $\mu,\nu$, introduced in Definition~\ref{definizione-mu-nu}.  By means of $\mu$ and $\nu$, we can compute the $*$-exponential in an easier way and thus compare $\exp_*(f)$ and $\exp\circ f$ for $f\in\So(\Omega)$. Lastly, we prove that for any slice regular function $g$ without non-real isolated zeroes and such that $g^s\not\equiv 0$ has a square root $\tau$, the quotient $\tau^{-1}g$ is a well defined slice regular function.

In Section~\ref{mu} we investigate in full details the behaviour of the functions $\mu$ and $\nu$ introduced in the previous section.  For $n\in\N$, we define a family of circular domains $\mathcal D_n\subset \HH$, showing that the restriction of the map $\mu$ to $\mathcal{D}_{0}$ is biregular onto $\HH\setminus(-\infty,-1]$, while for any positive $n$, the restriction of the map $\mu$ to $\mathcal{D}_{n}$ is biregular onto $\HH\setminus((-\infty,-1]\cup[1,+\infty))$. This allows us to introduce the inverse $\varphi:\HH\setminus(-\infty,-1]\to\mathcal D_0$ of $\mu|_{\mathcal D_0}$ which is a never-vanishing function on $\mathcal D_0$ and
will be extensively used in Section~\ref{the-core}.

Section~\ref{first_existence_results} contains a first existence result for the $*$-logarithm: namely, we prove that any one-slice preserving never-vanishing function $g$ defined on a slice contractible domain $\Omega$ has a $*$-logarithm which preserves the same slice. Moreover, we show that, if $\Omega$ is slice and $g$ is positive on $\Omega\cap \R$, then there exists a \textit{unique} slice-preserving $*$-logarithm of $g$, while if $\Omega$ is product we can always find a slice preserving $*$-logarithm of a slice preserving function, but in this case the $*$-logarithm is \textit{never} unique.
This statement allows us to prove the existence of a solution of a cosine-sine problem, in strong analogy with the case of holomorphic functions. Indeed, given a slice contractible domain $\Omega$ and $a_{0},a_{1}\in\So_{\R}(\Omega)$ such that $a_{0}^{2}+a_{1}^{2}\equiv 1$, we show that there exists $\gamma\in\So_{\R}(\Omega)$ such that
$\cos_{*}(\gamma)=a_{0}$ and $\sin_{*}(\gamma)=a_{1}$.
As a consequence of this proposition we are able to classify zero divisors with identically zero real part on product domains.

In Section~\ref{uniqueness-results} we start by studying the uniqueness problem for the $*$-logarithm on slice preserving functions and we then turn to the general case. The main point is that we have to consider two different situations: the first when we deal with slice preserving functions and the second one when we take into account non slice preserving functions. 

We underline that the results we obtain in this section do not depend on any topological condition on the domain $\Omega$.

For the first case we need to define the set
$$ 
\mathfrak{N}(\Omega):=\{f\in\So(\Omega)\,|\,\exists\, m\in\Z\setminus\{0\}, f_{v}^{s}=m^{2}\pi^{2}\}\cup\So_{\R}(\Omega),
$$
whose elements have the property that the symmetrized function of their vectorial part always has a square root. This allows us to state the following 
\begin{theorem}
Let $h,\tilde{h}\in\So(\Omega)$ be such that $\exp_{*}(h)=\exp_{*}(\tilde{h})\in\So_{\R}(\Omega)$.

\noindent $\bullet$ If $\Omega$ is a slice domain, then $h_{0}\equiv \tilde{h}_{0}$, $h,\tilde{h}\in\mathfrak{N}(\Omega)$ and $\sqrt{h_{v}^{s}}\equiv \sqrt{\tilde{h}_{v}^{s}}$ \textnormal{(mod $2\pi$)}.

\noindent $\bullet$ If $\Omega$ is a product domain,  then there exists $n\in\Z$ such that
$h_{0}=\tilde{h}_{0}+\pi n\mathcal{J}$, where $\mathcal{J}$ is the function given in Definition~\ref{mathcalJ}. Moreover
$h,\tilde{h}\in\mathfrak{N}(\Omega)$ and $\sqrt{h_{v}^{s}}\equiv \sqrt{\tilde{h}_{v}^{s}}+n\pi$ \textnormal{(mod $2\pi$)}.

\end{theorem}

When the $*$-exponential of the functions we are considering do not belong to $\So_\R(\Omega)$, the conclusion we obtain is quite different.

\begin{theorem}
Let $h,\tilde{h}\in\So(\Omega)$ be such that $h\neq\tilde h$ and $\exp_{*}(h)=\exp_{*}(\tilde{h})\not\in\So_{\R}(\Omega)$.

\noindent $\bullet$ If $\Omega$ is a slice domain, then 
$h_{0}\equiv \tilde{h}_{0}$, 
both $h_{v}$ and $\tilde h_{v}$ 
have no non-real isolated zeroes, both $h_{v}^{s}$ and $\tilde{h_{v}^{s}}$ have a square root on $\Omega$ and there exist $m\in\Z\setminus\{0\}$, $\alpha\in\So_\R(\Omega)$ and $H_v\in\So(\Omega)$ with $H_v^s\equiv1$  such that $h_v=\alpha H_v$ and 
$\tilde{h}_{v}=(\alpha +2\pi m)H_v=h_{v}+2\pi m H_v$, so that
$$
\tilde h=h+2\pi m H_v.
$$
\noindent $\bullet$ If $\Omega$ is a product domain,  then one of the following holds
\begin{enumerate}
\item there exists $n\in\Z\setminus\{0\}$ such that
$$\tilde h=h+2\pi n\mathcal{J}.$$
\item both $h_v$ and $\tilde h_v$ are not zero divisors and have no non-real isolated zeroes, both $h_{v}^{s}$ and $\tilde{h_{v}^{s}}$ have a square root on $\Omega$ and there exists
$n,m\in\mathbb Z$ such that $m\neq0$   and  
$m\equiv n$ \textnormal{(mod $2$)}, $\alpha\in\So_\R(\Omega)$ and $H_v\in\So(\Omega)$ with $H_v^s\equiv1$, such that 
$h_v=\alpha H_v$ and 
$$\tilde{h}=h_{0}+\pi n\mathcal{J}+(\alpha +\pi m)H_v=h+\pi (n\mathcal{J}+ m H_v).$$
\end{enumerate}
\end{theorem}

If the domain $\Omega$ is slice and the function $h_{v}$ has a non real isolated zero, the above theorem gives an unexpected uniqueness result for the $*$-logarithm.

\begin{corollary}
Let $\Omega$ be slice
and $h\in\So(\Omega)$ be such that $\exp_{*}(h)\not\in\So_{\R}(\Omega)$.
If $h_{v}$ has a non real isolated zero, then $\exp_{*}(h)=\exp_{*}(\tilde h)$ if and only if $\tilde h\equiv h$.
\end{corollary}

Last Section contains the most important existence results of our paper. First of all we show that, if the domain is slice-contractible, we can always limit ourselves to look for the $*$-logarithm of a slice regular function $g$ such that $g^s\equiv1$.

We then get rid of the case when $g_v$ is a zero divisor, showing that  if  $\Omega$ is a slice contractible domain and $g$ is a never vanishing function such that $g_v$ is a zero-divisor, then there exists a $*$-logarithm of $g$.

Next we turn to the case when $g_v$ is not a zero divisor (which is always the case when $\Omega$ is slice); under this hypothesis
we find a necessary condition for the existence of a $*$-logartithm of a function $g$ whose symmetrized function is identically equal to $1$. Indeed, if there exists a $*$-logarithm of such a $g$ we have that
\begin{enumerate}
\item if $\Omega$ is a slice domain and $q_{0}$ is a non-real isolated zero of $g_{v}$, then $g(q_{0})=1$;
\item if $\Omega$ is a product domain and $q_{0}, q_{1}$ are non-real
isolated zeroes of $g_{v}$, then, either $g(q_{0})=g(q_{1})=1$
or $g(q_{0})=g(q_{1})=-1$.
\end{enumerate}
Thus the non real isolated zeroes of $g_v$ are the true obstruction we have to overcome in order to get the existence of a $*$-logarithm. We notice that, in the case of a slice domain, the above conclusion gives the following unexpected negative results for the existence of a $*$-logarithm of a function.

\begin{enumerate}
\item Let $\Omega$ be a slice domain and $f\in\So(\Omega)$ be such that $f_{v}$ has a non-real isolated zero. Then $-\exp_{*}(f)$ has no $*$-logarithm. 
\item Let $\Omega$ be a slice contractible slice domain and $g\in\So(\Omega)$ be a never-vanishing function such that $g_{v}$ has a non-real isolated zero. Then at least one 
between $g$ and $-g$ has no $*$-logarithm. 
\end{enumerate}
%
%
%
%

The fact that the non real isolated zeroes of $g_v$ are the genuine obstruction for the existence of a global $*$-logarithm of a function is confirmed by the following statement.

\begin{theorem}\label{thm14}
Let $\Omega$ be slice contractible. Then any never vanishing $g\in\So(\Omega)$ such that $g_{v}$ has no non-real isolated zeroes has a $*$-logarithm. 
\end{theorem}

We now consider the problem near the non-real isolated zeroes of $g_{v}$:
by means of the inverse of the function $\mu$, we are able to build a $*$-logarithm of $g$ near these points (to be more precise, afar from the points where $g_0$ takes values in $(-\infty, -1]$). 

\begin{theorem}\label{thm15}
Let $\Omega$ be slice contractible and $g\in\So(\Omega)$. If $g^s\equiv 1$ and for any $q_{0}\in\Omega$ that is a non-real isolated zero of $g_{v}$
we have that $g_{0}(q_{0})=1$, then,
on every connected component of $\Omega\setminus g_{0}^{-1}((-\infty,-1])$,
there exists a $*$-logarithm of $g$.
%
\end{theorem}

We conclude our investigation by gluing the solutions obtained in Theorem~\ref{thm14}
where $g_v$ has no non-real isolated zeroes with the ``patches" given by
 Theorem~\ref{thm15}, thus obtaining our main theorem. 
 The topological difficulties we have to deal with, prevent us to obtain a clear statement in the general case of a slice contractible domain, so we prefer to give a result which pays something to generality in order to obtain a simpler formulation.
 
\begin{theorem}
Let $\Omega$ be one among $\mathbb{B}$, $\HH$ or $\mathbb{D}$. Let $g\in\So^{1}(\Omega)$ be such that 
\begin{itemize}
\item $g_{v}$ has
a finite number of non-real isolated zeros $\{q_{1},\dots,q_{N}\}$;
\item $g_{0}(q_{\ell})=1$ for all $\ell=1,\dots, N$;
\item the union $\SF_{q_{1}}\cup\dots\cup\SF_{q_{N}}$ is contained in a unique connected
component $\mathcal{U}$ of $\Omega\setminus g_{0}^{-1}((-\infty,-1])$.
\end{itemize}
If for some $I\in\SF$ (and hence for any) the set $\mathcal{U}_{I}^{+}=\mathcal{U}\cap\C_{I}^{+}$ is convex and $\mathcal U$ is slice if $\Omega$ is, then there exists a slice regular $*$-logarithm of $g$. 
\end{theorem}

\section{Preliminary results}\label{preliminaries}

In this section we recall the basic definitions and results we will use in the following. We also prove a couple of basic, yet new, results.

We start with some general fact about quaternions. The real skew algebra of 
quaternions is defined as 
$$
\HH:=\{q=q_{0}+q_{1}i+q_{2}j+q_{3}k\,|\,q_{\ell}\in\R,\, \ell=0,1,2,3,\, i^{2}=j^{2}=k^{2}=-1,\,ij=-ji=k\};
$$
hence, $i,j$ and $k$ are imaginary units. From these three elements it is possible to construct the sphere of imaginary units defined as
$$\SF:=\{\alpha_{1}i+\alpha_{2}j+\alpha_{3}k\,|\,\alpha_{1}^{2}+\alpha_{2}^{2}+\alpha_{3}^{2}=1\}=\{I\in\HH\,|\,I^{2}=-1\}.$$

By means of the standard conjugation $q=q_{0}+q_{1}i+q_{2}j+q_{3}k\mapsto q^{c}=q_{0}-(q_{1}i+q_{2}j+q_{3}k)$, we can split any quaternion in its real and vector part $q=q_{0}+q_{v}$, where $q_{v}=(q-q^{c})/2$. The set of purely imaginary quaternions will be denoted by $Im(\HH)=\{q\in\HH\,|\,q_{0}=0\}$.
Another possible splitting is given as follows: any $q\in\HH\setminus\R$ can be uniquely written 
as $q=\alpha+I\beta$, where $\alpha=q_{0}$, $\beta=||q_{v}||>0$ and $I=q_{v}/||q_{v}||\in\SF$. Such a decomposition gives the fundamental description of $\HH$ underlying the theory of slice regular functions. Before presenting it, we need
to introduce a few more definitions to specify the class of domains where 
slice regular functions are defined.
For any imaginary unit $I\in\SF$ we denote by $\C_{I}$ the complex \textit{slice} 
spanned by $1$ and $I$ over the reals, i.e.
$$
\C_{I}:=\text{Span}_{\R}(1,I).
$$
We will also consider closed \textit{semislices}
$$\C_{I}^{+}=\{\alpha+I\beta\,|\,\alpha,\beta\in\R,\,\beta\geq 0\},$$
and, for any $q\in\HH\setminus\R$  we denote by $\C_{q}$ the unique slice containing $q$.
Moreover, if $\Omega\subset \HH$, then we write $\Omega_{I}=\Omega\cap\C_{I}$ and $\Omega_{I}^{+}=\Omega\cap \C_{I}^{+}$.
  
Given any $q\in\HH$, we define its sphere $\SF_{q}$ as 
$$
\SF_{q}:=\{q_{0}+||q_{v}||J\,|\,J\in\SF\},
$$  
where, trivially, if $q\in\R$, then $q_{v}=0$ and $\SF_{q}=\{q\}$.
For any $\Omega\subset \HH$, its \textit{circularization} is given by
$$
\bigcup_{q\in\Omega}\SF_{q};
$$
a domain $\Omega\subset\HH$ will be called \textit{circular} if it coincides with its circularization. A circular domain is said to be \textit{slice} if $\Omega\cap\R\neq \emptyset$, while is called \textit{product} if $\Omega\cap\R=\emptyset$.
Notice that if $\Omega$ is a product domain, then $\Omega=\Omega_{I}^{+}\times \SF$, for any $I\in\SF$, thus explaining the origine of its naming.

We are now ready to give our main topological definition.
\begin{definition}
A circular subset $\Omega\subset\HH$ is said to be \textit{slice contractible}
if, for some $I\in\SF$ (and then any), $\Omega_{I}$ is a simply connected domain
if $\Omega$ is slice, and $\Omega_{I}^{+}$ is a simply connected domain if
$\Omega$ is product.
\end{definition}

From now on we denote by $\mathbb{B}\subset\HH$
 the open unit ball centered at the origin and by $\mathbb{D}$ the ``solid torus'' 
$$
\mathbb{D}=\{\alpha+I\beta\in\HH\,|\,\alpha^{2}+(\beta-2)^{2}<1,\,I\in\SF\},
$$
which are both slice contractible domains, the first one slice and the second one product. 
\begin{remark}
We notice that for a slice domain $\Omega$ being simply connected does not mean being slice contractible, in general. Indeed, if $\Omega=\mathbb{B}\setminus\{0\}$, 
then $\Omega$ is simply connected because it retracts on $S^{3}$, while $\Omega_{I}$ is not simply connected because it is a ``flat'' punctured disc in $\C_{I}$.
\end{remark}

The following definitions identify the functions we will work with (see~\cite{G-P} for an introduction to the topic).

\begin{definition}
Let $\Omega$ be a circular domain in $\HH$ and set $D=\{\alpha+\sqrt{-1}\beta\,|\,\alpha+I\beta\in\Omega\}\subset\C$.
We say that a function $F:D\to\HH\otimes_{\R} \C\simeq\HH\oplus\imath\HH$ is a {\sl stem} function if 
the equality $F(\overline z)=\overline{F(z)}$, where $\overline{p+\imath q}=p-\imath q$, holds for any $p+\imath q\in\HH\otimes_{\R} \C$.
A {\sl slice} function $f:\Omega\to \HH$ is a function induced by a stem function $F=F_0+\imath F_1:D\to\HH\otimes_{\R} \C$ in the following way:
$f(\alpha+\beta I)=F_0(\alpha+\sqrt{-1} \beta)+I F_1(\alpha+\sqrt{-1} \beta)$. Such a function will also be denoted by $f=\mathcal I(F)$.
\end{definition}

Ghiloni and Perotti~\cite{G-P} proved that a slice function $f$ is induced by a unique stem function $F = F_{0} + \imath F_{1}$, given by 
$F_{0}(\alpha+\sqrt{-1} \beta) = \frac{1}{2}(f(\alpha+I\beta)+f(\alpha-I\beta))$ and $F_{1}(\alpha+\sqrt{-1} \beta) = -\frac{1}{2}I(f(\alpha+I\beta)-f(\alpha-I\beta))$ for
any $I \in \SF$.
In particular, if $f=\mathcal{I}(F=F_{0}+\imath F_{1})$ we define its \textit{slice conjugate} as the function $f^{c}:=\mathcal{I}(F^{c}):\Omega\to\HH$, where $F^{c}=F_{0}^{c}+\imath F_{1}^{c}$, where $F_{\ell}^{c}(q)=(F_{\ell}(q))^{c}$ for $\ell=0,1$.

\begin{definition}
Let $\Omega\subset \HH$ be a circular domain. A {\sl slice} function $f=\mathcal I(F):\Omega\to \HH$ is slice {\sl regular} if $F$ is holomorphic with respect to the natural complex structures of $\C$ and $\HH\otimes_{\R} \C$. We denote by $\So(\Omega)$ the set of all slice regular functions on $\Omega$ with its natural structure of right $\HH$-module.
\end{definition}

If $\Omega$ is slice, the notion of slice regularity coincides with \textit{Cullen regularity} (see~\cite{G-S-St}). A useful result for slice-regular functions is the following

\begin{proposition}[Representation Formula] Let $f\in\So(\Omega)$ and let $\alpha+\beta J\in\Omega$. For all $I\in\SF$ we have 
$$
f(\alpha+J\beta)=\frac{1-JI}2f(\alpha+I\beta)+\frac{1+JI}2f(\alpha-I\beta). 
$$
\end{proposition}
A useful consequence of the Representation Formula is the fact that if $f_{I}:\Omega_{I}\to\HH$  is a holomorphic function with respect to the left multiplication by $I$, then there exists a unique slice-regular function $f:\Omega\to\HH$
such that $f_{I}=f_{|\Omega_{I}}$.  Such a function $\text{ext}(f_{I})$ will be called the {\sl regular extension of $f_I$}, (see~\cite{G-S-St}). From now on if 
$f\in\So(\Omega)$ we denote by $f_{I}$ its restriction to $\Omega_{I}$. 

The strong relation between holomorphicity and slice regularity appears also
in the following result obtained by merging \cite[Theorem 1.12]{G-S-St} and \cite[Theorem 3.6]{A-CVEE}.
\begin{proposition}[Identity Principle]
Given $f,g\in\So(\Omega)$
\begin{itemize}
\item if $\Omega$ is slice and for some $I\in\SF$, the functions $f$ and $g$ coincide on a subset of $\Omega_{I}$ having an accumulation point, then $f\equiv g$;
\item if $\Omega$ is product and for some $I\in\SF$, the functions $f$ and $g$ coincide on a subset of $\Omega_{I}$ having an accumulation point in $\Omega_{I}^{+}$ and an accumulation point in $\Omega_{I}^{-}$, then $f\equiv g$.
\end{itemize}
\end{proposition}
 
The following two classes of regular functions are of particular interest for the theory.
\begin{definition}\label{slice-preserving-definition}
Given $f\in\So(\Omega)$, we say that $f$ is \textit{one slice preserving} if 
$f(\Omega_{I})\subset\C_{I}$ for some $I\in\SF$ and $f$ is \textit{slice preserving} 
if $f(\Omega_{I})\subset\C_{I}$ for all $I\in\SF$.  The set of slice preserving function on $\Omega$ will be denoted by $\So_{\R}(\Omega)$, while the set of functions that preserve $\Omega_{I}$ will be denoted by $\So_{I}(\Omega)$.
\end{definition}


\begin{remark}
Notice that a function $f$ is slice preserving if and only if it preserves two different slices, if and only if it is intrinsic, i.e., $f(q^{c})=(f(q))^{c}$, for all $q$ in its domain $\Omega$.
Moreover, if $\Omega$ is slice, $f$ is slice preserving if and only if 
$f(\Omega\cap\R)\subset\R$.
\end{remark}

We now introduce an interesting function with will be widely used in the course
of the paper.
\begin{definition}\label{mathcalJ}
Let us define the regular (slice preserving) function
$\mathcal{J}:\HH\setminus\R\to \SF\subset\HH$ as
$$
\mathcal{J}(q_{0}+q_{v})=\frac{q_{v}}{||q_{v}||}.
$$
\end{definition}
Equivalently if $q=\alpha+I\beta\in\HH\setminus\R$, with $\beta>0$, then $\mathcal{J}(q)=I$.

We now turn to the algebraic structure of $\So(\Omega)$.
It is a well known fact that in general the pointwise product of two slice-regular functions is no more slice. Nonetheless, this problem can be overcome by defining the following non-commutative product (see~\cite{C-G-S-St}).

\begin{definition}
Let $f,g\in\So(\Omega)$. We define the $*$-\textit{product} of $f$ and $g$ as the slice regular function $f*g:\Omega\to\HH$ given by
$$(f*g)(q):=\begin{cases}
0,&\mbox{if } f(q)=0,\\
f(q)g(f(q)^{-1}qf(q)),&\mbox{if }f(q)\neq 0.
\end{cases}
$$
\end{definition}

In some specific case, the $*$-product coincides with the pointwise product and in some more it turns to be commutative.
\begin{remark}
Let $f,g\in\So(\Omega)$. 
If $\Omega$ is slice, then $f*g=f\cdot g$ on $\Omega\cap\R$.
If $f$ is slice preserving, then $f*g=g*f=fg$. 
If 
both functions preserve the same slice $\Omega_{I}$, then $f*g=\text{ext}(f_{I}\cdot g_{I})=\text{ext}(g_{I}\cdot f_{I})=g*f$.
\end{remark}

Given any slice regular function $f\in\So(\Omega)$, we then introduce its
\textit{symmetrized function} $f^{s}\in\So_{\R}(\Omega)$ defined as $f^{s}=f*f^{c}=f^{c}*f$.

The importance of the symmetrized function relies mainly in its connection with 
the zero set of $f$: in particular, the function $f$ is a zero divisor if and only if $f^{s}\equiv 0$ and the fact that $f$ has a non real isolated zero $q$, entails $f^{s}\equiv 0$ on $\SF_{q}$.

The symmetrized function allows us to define the $*$-inverse of a regular function.
Given $f\in\So(\Omega)$ with $f^{s}\not\equiv 0$, we define its $*$\textit{-inverse} as
$f^{-*}=(f^{s})^{-1}f^{c}$ which is slice regular outside the zero set of $f^{s}$.
In the following we will extensively use the following zero-product property~\cite[Proposition 5.18]{GPSalgebra}.
\begin{proposition}[Ghiloni-Perotti-Stoppato]\label{propGPS} Let $f,g\in\So(\Omega)$ be such that 
$f^{s}\not\equiv 0$. Then, $f*g\equiv 0$ implies $g\equiv 0$. In particular
\begin{itemize}
\item $f\in\So(\Omega)\setminus\{0\}$ is a zero divisor if and only if $f^{s}\equiv 0$;
\item if $\Omega$ is slice, then $\So(\Omega)$ is a domain;
\item if $f\in\So_{\R}(\Omega)\setminus\{0\}$, then $fg\equiv 0$ implies $g\equiv 0$.
\end{itemize}
\end{proposition}

Now assume $(1,I,J,K)$ is an orthonormal basis of $\HH$, thanks to 
\cite[Proposition 3.12]{C-GC-S} and~\cite[Lemma 6.11]{G-M-P}, any slice regular function
$f\in\So(\Omega)$ can be uniquely written as a sum
$f=f_{0}+f_{1}I+f_{2}J+f_{3}K$, where $f_{0},\dots, f_{3}\in\So_{\R}(\Omega)$,
thus giving to $\So(\Omega)$ the structure of a $4$-rank free module on $\So_{\R}(\Omega)$. For the convenience of what follows 
we call
$f_{0}$ the \textit{``real part''} of $f$ and $f_{v}=f-f_{0}$ the  \textit{``vector part''} of $f$.
We also introduce the following two operators: let $f, g\in\So(\Omega)$, then
$$
\langle f,g\rangle_{*}:=(f*g^{c})_{0},\qquad f\pv g:=\frac{f*g-g*f}{2}.
$$

The following result summarizes a series of properties of the $*$-product obtained via the above interpretation of the multiplicative structure of $\So(\Omega)$ (see~\cite{A-dF}).
\begin{proposition}\label{product-properties}
Let $f=f_{0}+f_{1}I+f_{2}J+f_{3}K, g=g_{0}+g_{1}I+g_{2}J+g_{3}K\in\So(\Omega)$. Then
\begin{itemize}
\item $f^{c}=f_{0}-(f_{1}I+f_{2}J+f_{3}K)$, $f_{0}=\frac{f+f^{c}}{2}$ and
$f_{v}=\frac{f-f^{c}}{2}$;
\item $f*g=f_{0}g_{0}-\langle f_{v},g_{v}\rangle_{*}+f_{0}g_{v}+g_{0}f_{v}+f_{v}\pv g_{v}$;
\item $f^{s}=f_{0}^{2}+f_{1}^{2}+f_{2}^{2}+f_{3}^{2}$; in particular $f^{s}\geq 0$ on $\Omega\cap \R$;
\item $f_{v}*f_{v}=-f_{v}*(-f_{v})=-f_{v}*f_{v}^{c}=-f_{v}^{s}$.
\end{itemize}
\end{proposition}

Following~\cite{C-S-St-2} we define the following operators on $\So(\Omega)$.

\begin{definition}
Let $f\in\So(\Omega)$. We set
$$\exp_*(f)=\sum_{n\in\mathbb{N}}\frac{f^{*n}}{n!},\quad \cos_*(f)=\sum_{n\in\mathbb{N}}\frac{(-1)^{n}f^{*(2n)}}{(2n)!},\quad\sin_*(f)=\sum_{n\in\mathbb{N}}\frac{(-1)^{n}f^{*(2n+1)}}{(2n+1)!}.$$
\end{definition}

\begin{remark}\label{caso-slice-pres}
If $\Omega$ is slice, then $\exp_{*}(f)=\exp\circ f$ on $\Omega\cap\R$ and the same holds for $\cos_{*}$ and $\sin_{*}$, i.e. $\cos_{*}(f)=\cos\circ f$ and $\sin_{*}(f)=\sin\circ f$ on $\Omega\cap \R$.
If $f\in\So_{\R}(\Omega)$ then
$\exp_{*}(f)=\exp\circ f= \exp(f)$,
$\cos_{*}(f)=\cos\circ f= \cos(f)$ and $\sin_{*}(f)=\sin\circ f= \sin(f)$.
Moreover, if $f\in\So_{I}(\Omega)$ then $\exp_{*}(f)=\text{ext}(\exp(f_I))$,
$\cos_{*}(f)=\text{ext}(\cos(f_I))$ and $\sin_{*}(f)=\text{ext}(\sin(f_I))$.
\end{remark}

We introduce the following definition in order to restate some of the contents of~\cite{A-dF}.

\begin{definition}\label{definizione-mu-nu}
We denote by $\mu,\nu:\HH\to\HH$ the following slice preserving entire functions
$$
\mu(q)=\sum_{n\in\mathbb{N}}\frac{(-1)^{n}q^{m}}{(2m)!},\qquad
\nu(q)=\sum_{n\in\mathbb{N}}\frac{(-1)^{n}q^{m}}{(2m+1)!}.
$$
\end{definition}

\begin{remark}
We notice that 
$$\mu(q^{2})=\cos(q),\qquad\nu(q^{2})q=\sin(q),$$ for all $q\in\HH$.
In particular, 
$$\mu(\pi^{2}n^{2})=(-1)^{n},\, \text{for all}\, n\in\mathbb{Z},\quad \nu(0)=1\quad
\text{and}\quad\nu(\pi^{2}n^{2})=0\, \text{if}\, n\in\mathbb{Z}\setminus\{0\}.$$
Moreover, for any $q\in\HH$, the following equality holds
\begin{equation}\label{sommaquadrati}
\mu^{2}(q)+\nu^{2}(q)q\equiv 1.
\end{equation}
Indeed for any $q\in\R^{+}$, choose $t\in\R\setminus\{0\}$ such that $q=t^{2}$, then $\mu(q)=\mu(t^{2})=\cos(t)$, $\nu(q)=\nu(t^{2})=\frac{\sin(t)}{t}$. Thus $\mu^{2}(q)+\nu^{2}(q)q=\cos^{2}(t)+\frac{\sin^{2}(t)}{t^{2}}t^{2}=\cos^{2}(t)+\sin^{2}(t)=1$. By the Identity Principle we are done.
\end{remark}

Next proposition collects several features of the $*$-exponential (see~\cite{A-dF},
where a slightly different notation is used).

\begin{proposition}
Let $f,g\in\So(\Omega)$, then we have the following equalities
\begin{align}
\exp_{*}(-f)&=(\exp_{*}(f))^{-*};\label{inverse}\\
(\exp_{*}(f))^{s}&=\exp(2f_{0});\nonumber\\
\label{munu}
\exp_{*}(f)&=\exp_{*}(f_{0})\left(\mu(f_{v}^{s})+\nu(f_{v}^{s})f_{v}\right);\\
\exp_{*}(f+g)&=\exp_{*}(f)*\exp_{*}(g),\qquad \mbox{if}\quad f*g=g*f.\nonumber
\end{align}
\end{proposition}
In particular Equality~\eqref{inverse} shows that $\exp_{*}(f)$ is never-vanishing on $\Omega$.

The following examples elucidate the behavior of the $*$-exponential in some
notable cases.

\begin{example}\label{exppiJ=-1}
As $\mathcal{J}$ defined in~\ref{mathcalJ} is slice preserving, then a straightforward computation gives $\mathcal{J}^{s}=\mathcal{J}^{2}\equiv -1$.
Therefore, we have
\begin{align*}
\exp_{*}(\pi\mathcal{J})&=\exp(\pi\mathcal{J})=\sum_{n\in\mathbb{N}}\frac{(\pi\mathcal{J})^{n}}{n!}=\sum_{m\in\mathbb{N}}\frac{(-1)^{m}\pi^{2m}}{(2m)!}+\sum_{m\in\mathbb{N}}\frac{(-1)^{m}\pi^{2m+1}}{(2m+1)!}\mathcal{J}\\
&=\cos(\pi)+\sin(\pi)\mathcal{J}=-1.
\end{align*}
\end{example}

\begin{example}
If $g_{v}$ a zero divisor, then $g_{v}^{s}\equiv 0$ and thus,
$$
\exp_{*}(g_{v})=\mu(g_{v}^{s})+\nu(g_{v}^{s})g_{v}=\mu(0)+\nu(0)g_{v}=1+g_{v}.
$$
\end{example}

\begin{example}\label{esempi-exp-composizione}
Let $f(q)=i+qj$. As $f_{0}\equiv 0$ and $f_v^s=1+q^2$ we have
$$\exp_*(f)=\mu(q^2+1)+\nu(q^2+1)(i+qj).$$
In particular for any $q\in\SF$ we have $q^2+1=0$, so $(\exp_*(f))(q)=1+i+qj$,
giving in particular $(\exp_*(f))(j)=i$. 
Nonetheless,  $(\exp\circ f)(j)=e^{f(j)}=e^{i-1}$.
\end{example}
\begin{example}
If $f(q)=\pi\cos(q)i+\pi\sin(q)j$, again $f_{0}\equiv 0$ and
$f_{v}^{s}\equiv \pi^{2}$, so 
$$\exp_*(f)=\mu(\pi^2)+\nu(\pi^2)\pi(cos (q) i+ \sin (q) j)\equiv-1.$$
\end{example}
Lastly, we compute $\exp_{*}$ on one-slice preserving functions.
\begin{remark}
Notice that if $f$ is $\mathbb C_I$-preserving for some $I\in\SF$, then we have $f_v=f_1I$ with  $f_1$  slice preserving. This entails
$$\exp_*(f)=\exp(f_0)(\cos(f_1)+\sin(f_1)I).$$  
Indeed,  $\mu(f_v^s)=\mu(f_1^2)=\cos (f_1)$, 
$\nu(f_v^s)f_v=\nu(f_1^2)f_1I=\sin(f_1)I$.
\end{remark}

We now introduce the notion of $*$-logarithm.

\begin{definition} 
We set
\begin{align*}
\mathcal{R}^{*}(\Omega)&:=\{g\in\So(\Omega)\,|\, g\mbox{ is never vanishing}\},\\
\mathcal{R}^{1}(\Omega)&:=\{g\in\So(\Omega)\,|\, g^{s}\equiv 1\}.
\end{align*}
Given $g\in\So^{*}(\Omega)$, a function $f\in\So(\Omega)$ such that $\exp_{*}(f)=g$ is said to be a $*$-\textit{logarithm} of $g$.
\end{definition}

Notice that the elements in $\mathcal{R}^{*}(\Omega)$ act by conjugation both on $\So^{1}(\Omega)$ and on 
the set of functions having a $*$-logarithm.

\begin{remark}\label{conjugate}
If $h\in\So^{1}(\Omega)$ and $\chi\in \mathcal{R}^{*}(\Omega)$, then
$(\chi^{-*}*h*\chi)^{s}=\chi^{-s}*h^{s}*\chi^{s}=\chi^{-s}*\chi^{s}=1$.
Moreover, if $f\in\mathcal{R}(\Omega)$ is a $*$-logarithm of $h$, then
a trivial computation shows that $\chi^{-*}*f*\chi$ is a $*$-logarithm of 
$\chi^{-*}*h*\chi$ as $\exp_{*}(\chi^{-*}*f*\chi)=\chi^{-*}*\exp_{*}(f)*\chi$.
%
%
\end{remark}

We conclude this section by showing that a slice regular function without non-real isolated zeroes can be factorized as the product of a slice preserving function and
a slice regular function in $\So^{1}(\Omega)$.

\begin{proposition}\label{quotient} Let $g\in\So(\Omega)$ be such that $g$ has no non-real isolated zeroes and
$g^s$ has a square root $\tau\in\So_\R(\Omega)\setminus\{0\}$. Then $\tau^{-*}*g=g/\tau$ is a well-defined slice regular function on $\Omega$ which belongs to $\So^{1}(\Omega)$. 
\end{proposition}

\begin{proof} Since $\tau$ is a slice preserving function, outside the zero set of $\tau$ the function $\tau^{-*}*g$ is a well defined slice-regular function which coincides with the pointwise product $\tau^{-1}g=g/\tau$.

Then we are left to define the function $\tau^{-1}g$ at the zeroes of $\tau$. Since $g$ is not a zero divisor, then it only has real isolated zeroes and spherical isolated zeroes (non-real isolated zeroes are ruled out by the assumption), so the zero set of $\tau$ coincides with the zero set of $g$ (and of $g^s$).

If $x_0\in\Omega\cap \R$ is a zero of $g$, we choose a ball $U_{x_0}\subset\Omega$ centered at $x_0$ on which  $g$  vanishes at $x_0$ only. By Theorem 3.36 in~\cite{G-S-St} we can write $g(q)=(q-x_0)^m \tilde g(q)$ for a suitable $\tilde g$ slice-regular and never-vanishing on $U_{x_0}$. 
Then we have $g^s(q)=(q-x_0)^{2m}\tilde g^s(q)$. As $\tilde g^s$ is never-vanishing on $U_{x_0}$ and it is strictly positive on $U_{x_0}\cap \R$, then there exists $\alpha\in\So_\R(U_{x_0})$ such that $\alpha^2=\tilde g^s$.
A trivial computation shows that the slice-regular function $\beta(q)=(q-x_0)^m\alpha(q)$ is a square root of $g^s$ on $U_{x_0}$. By the identity principle, there are only two slice-preserving square roots of $g^s$ on $U_{x_0}$, so either 
$\beta=\tau$ or $\beta=-\tau$; up to a change of sign of $\alpha$ (and thus of $\beta$) we can suppose $\beta=\tau$.

On $U_{x_0}$ we consider the function $\tilde g(q)/\alpha(q)$ (which is well defined because $\alpha$ is never-vanishing and slice preserving on $U_{x_0}$). Moreover on $U_{x_0}\setminus\{x_0\}$ we have $\tilde g(q)/\alpha(q)=((q-x_0)^m \tilde g(q))/((q-x_0)^m\alpha(q))=g(q)/\tau(q)$, so that the two functions truly coincide on $U_{x_0}\setminus\{x_0\}$ and define a slice regular function on $U_{x_0}$.

Analogously, if $g$ has a spherical zero at $\SF_{q_0}$, then in a ``toric'' neighbourhood $U_{q_0}$ of such a sphere we can write $g(q)=\Delta_{q_0}(q)^m \tilde g(q)$ with $\tilde g$ slice regular and never vanishing on $U_{q_0}$. Again we find that $g^s(q)=\Delta_{q_0}(q)^{2m}\tilde g^s(q)$; as $\tilde g^s$ is never vanishing on $U_{q_0}$, it has a   slice-preserving square root $\alpha$ on that neighbourhood. Thus on $U_{q_0}$ the function $\beta(q)=\Delta_{q_0}(q)^m \alpha(q)$ is a square root of $g^s$ which coincides with $\tau$, up to a change of sign of $\alpha$ (and thus of $\beta$). 
On $U_{q_0}$ we consider $\tilde g(q)/\alpha(q)$ (which as above is well defined) and on $U_{q_0}\setminus\{\SF_{q_0}\}$ coincides with $g/\tau$, and we are done.
\end{proof}

If the domain $\Omega$ is slice-contractible, $g$ is not a zero divisor and has no non-real isolated zeroes, then
the existence of a square root $\tau\not\equiv0$ of $g^s$  is a consequence of Corollary 3.2 in~\cite{A-dF}.

\begin{corollary}\label{post-quotient} Let $\Omega$ be a slice-contractible domain. For any
$g\in\So(\Omega)$ which is not a zero divisor and has no non-real isolated zeroes, let us denote by $\tau\in\So_\R(\Omega)$ a square root of $g^s$. Then $\tau^{-*}*g=g/\tau$ is a well-defined slice regular-function on $\Omega$ which belongs to $\So^{1}(\Omega)$. 
\end{corollary}

%
%

\section{Behavior of the entire function $\mu$}\label{mu}

We now investigate the behavior of the slice preserving map $\mu$ on the quaternions. In particular, we will prove an invertibility result for the restriction of function $\mu$ to a family of subdomains of $\HH$ that will be used in Section~\ref{the-core}.

\begin{definition}
We define the following slice domains:
\begin{equation*}
\mathcal{D}_{0}:=\left\{x+yJ\,|\, x<\pi^{2}-\frac{y^{2}}{4\pi^{2}},\,J\in\SF\right\}\subset\HH,
\end{equation*}
and for any positive $n\in\N$,
\begin{equation*}
\mathcal{D}_{n}:=\left\{x+yJ\,|\, n^{2}\pi^{2}-\frac{y^{2}}{4n^{2}\pi^{2}}<x<(n+1)^{2}\pi^{2}-\frac{y^{2}}{4(n+1)^{2}\pi^{2}},\,J\in\SF\right\}\subset\HH.
\end{equation*}
For $n>0$ we also denote by 
$$
\Gamma_{n}:=\left\{x+yJ\,|\, x=n^{2}\pi^{2}-\frac{y^{2}}{4n^{2}\pi^{2}},\,J\in\SF\right\}\subset\HH
$$
the boundaries of the above domains. Indeed, we have $\partial\mathcal{D}_{0}=\Gamma_{1}$ and $\partial\mathcal{D}_{n}=\Gamma_{n}\cup\Gamma_{n+1}$, for any positive $n\in\N$.
\end{definition}
\begin{figure}[ht]
\includegraphics[width=10cm]{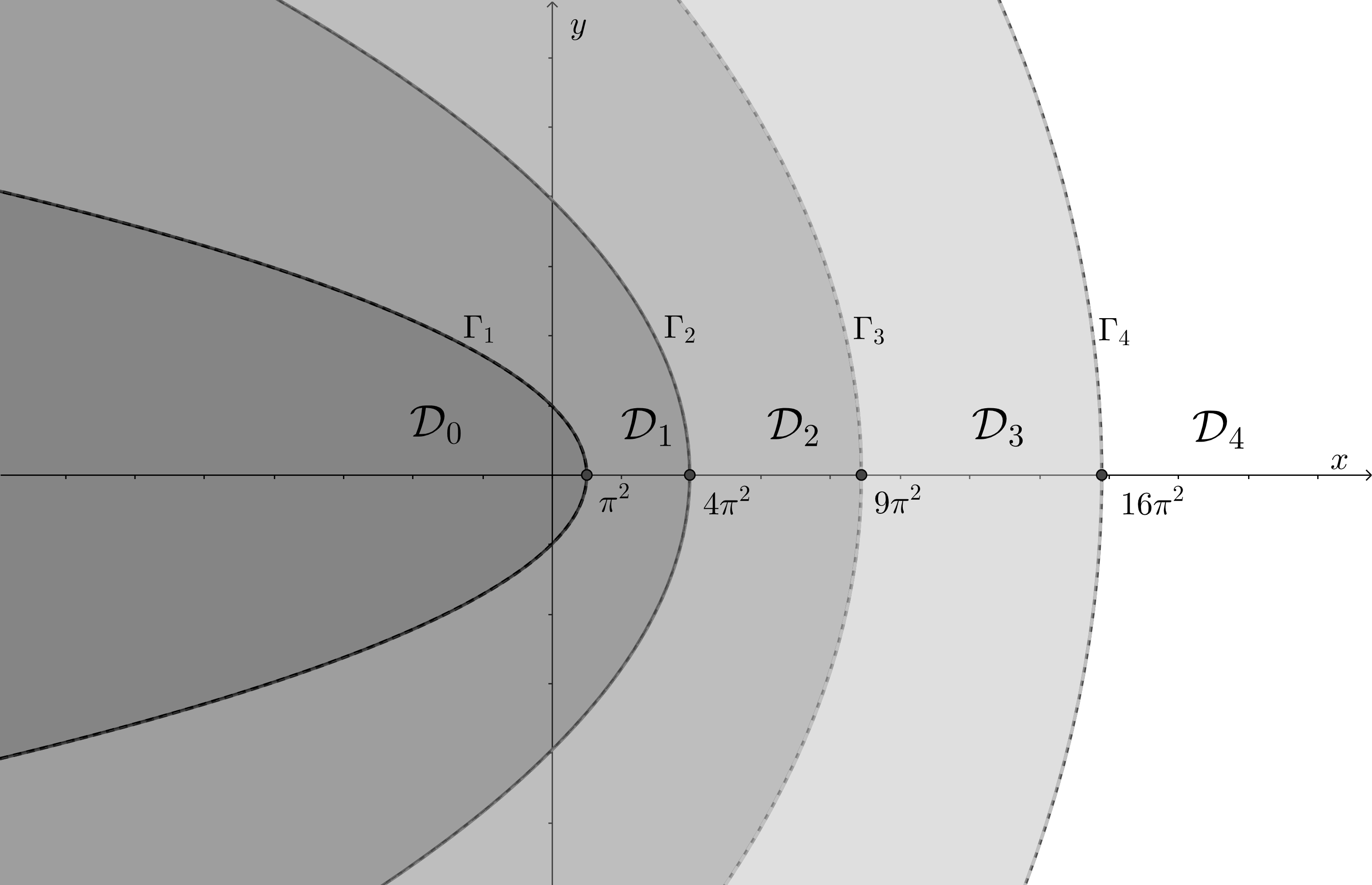}
\caption{The section of the sets $\mathcal{D}_{n}$ and $\Gamma_{n}$ on a fixed slice.}
\end{figure}

\begin{definition}
Let $\Omega$ be a slice domain and $f$ a non-constant slice preserving function on $\Omega$. We say that $f$ is \textit{biregular} on $\Omega$ if there exists
$g\in\So_{\R}(f(\Omega))$, such that $g\circ f=id_{\Omega}$ and
$f\circ g=id_{f(\Omega)}$. In such a case we call $g$ the \textit{biregular inverse} of $f$.
\end{definition}

\begin{theorem}\label{fundamentaldomain}
The restriction of the map $\mu$ to $\mathcal{D}_{0}$ is biregular onto $\HH\setminus(-\infty,-1]$. 
For any positive $n\in\N$, the restriction of the map $\mu$ to $\mathcal{D}_{n}$ is biregular onto $\HH\setminus((-\infty,-1]\cup[1,+\infty))$.
\end{theorem}

\begin{proof}
We first carry out the proof in the case of $\mathcal{D}_{0}$. 
Consider the set 
$$
\widetilde{\mathcal{D}_{0}}:=\{x+Jy\in\HH\,\,|\,\,0<x<\pi,\,\,J\in\SF\}.
$$
It is easily seen (working slice by slice), that the restriction  to $\widetilde{\mathcal{D}_{0}}$ of the function $s:\HH\to\HH$, defined as $s(q)=q^{2}$,
gives a bijection onto $\mathcal{D}_{0}\setminus(-\infty,0]$. 
Furthermore, working again slice by slice, the restriction to $\widetilde{\mathcal{D}_{0}}$ of the cosine function gives a bijection onto $\HH\setminus((-\infty,-1]\cup[1,+\infty))$.

For any $\xi\in\HH\setminus((-\infty,-1]\cup[1,+\infty))$ there exists a unique $t\in\widetilde{\mathcal{D}_{0}}$ such that
$\xi=\cos(t)$. Now $q=s(t)=t^{2}\in\mathcal{D}_{0}\setminus(-\infty,0]$ is the unique element of $\mathcal{D}_{0}\setminus(-\infty,0]$
which is mapped in $\xi$ by $\mu$; indeed we have $\mu(q)=\mu(t^{2})=\cos(t)=\xi$.

Now consider $\xi\in[1,+\infty)$. If $t\in\HH$ is such that $\cos(t)=\xi$, then there exist $k\in\N$ and $I\in\SF$ such that $a=2k\pi$, $b=\mbox{arccosh}(\xi)$ and $t=a+Ib$. Thus, $t^{2}=4k^{2}\pi^{2}-(\mbox{arccosh}(\xi))^{2}+4k\pi\mbox{arccosh}(\xi) I$ belongs to $\mathcal{D}_{0}$ if and only if $k=0$. In this particular case then $q=t^{2}$ is equal to $-(\mbox{arccosh}(\xi))^{2}\in(-\infty,0]$ which is the unique element in $q\in\mathcal{D}_{0}$ such that $\mu(q)=\xi$.

Then the function $\mu$ is a slice preserving bijection from $\mathcal{D}_{0}$ onto $\HH\setminus(-\infty,-1]$.
Let now $I\in\SF$. Since $\mu$ is slice preserving, then its restriction $\mu_{I}:\mathcal{D}_{0}\cap\C_{I}\to \C_{I}\setminus(-\infty,-1]$ is a holomorphic bijection and thus a biholomorphism. Then the inverse of  $\mu_{I}$ is a holomorphic function from $\C_{I}\setminus(-\infty,-1]$ to $\mathcal{D}_{0}\cap\C_{I}$, thus showing that the inverse of $\mu$ is slice regular.%

By using 
$$
\widetilde{\mathcal{D}_{n}}:=\{x+Jy\in\C\,\,|\,\,n\pi<x<(n+1)\pi,\,\,J\in\SF\},
$$
with $n>0$
in place of $\widetilde{\mathcal{D}_{0}}$, the above argument is easily adapted to prove the last
part of the assertion.

\end{proof}

\begin{figure}[ht]
\includegraphics[width=10cm]{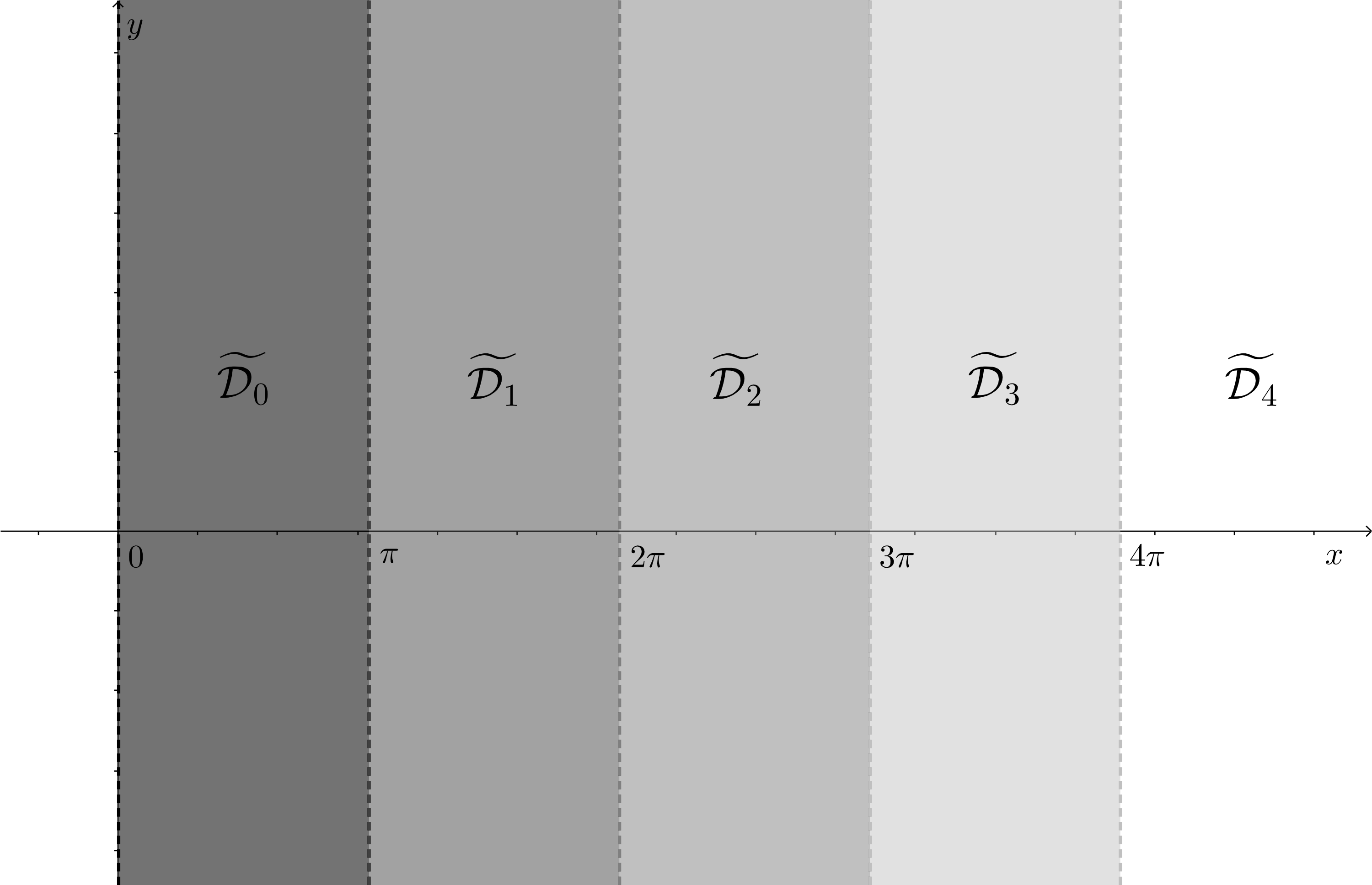}
\caption{The section of the sets $\widetilde{\mathcal{D}_{n}}$ on a fixed slice.}
\end{figure}

\begin{remark}
Notice that the boundaries of $\mathcal{D}_{0}$ and $\mathcal{D}_{n}$ are mapped by $\mu$ in either the left half line
$(-\infty,-1]$ or the right half line $[1,+\infty)$. Indeed we have
$$
\mu(\Gamma_{n})=\begin{cases}
(-\infty,-1],\quad\mbox{for}\,\,n\mbox{ odd}\\
[1,+\infty),\quad\mbox{for}\,\,n\mbox{ even}.
\end{cases}
$$
\end{remark}

Theorem~\ref{fundamentaldomain} allows us to give the following definition.
\begin{definition}\label{varphi}
We denote by $\varphi:\HH\setminus(-\infty,-1]\to\mathcal{D}_{0}$ the biregular  inverse of $\mu|_{\mathcal{D}_{0}}$.
\end{definition}

\begin{remark}
As an immediate consequence of the definition of $\varphi$ we have
$$\varphi\circ (\mu|_{\mathcal{D}_{0}})=id_{\mathcal{D}_{0}}\quad\mbox{and}\quad
\mu\circ\varphi=id_{\HH\setminus (-\infty,-1]}
$$
\end{remark}

\begin{remark}\label{zero}
In particular we observe that $\varphi (1)=0$ and that
the function $\nu$ is never vanishing on $\mathcal{D}_{0}$. The first assertion is trivial 
because $0$ is the unique point in $\mathcal{D}_{0}$ whose image via $\mu$ is $1$.
Moreover $q\in\HH$ is a zero of $\nu$ if and only if
$q\in\{n^{2}\pi^{2}\,|\,n\in\Z\setminus\{0\}\}$ and none of these points belong to $\mathcal{D}_{0}$.
\end{remark}

\section{Initial existence results}\label{first_existence_results}

We start now our discussion on the solvability of equation~\eqref{exp}. The first case we consider is when  the function $g$ preserves one slice; in this case, under suitable topological hypothesis on $\Omega$, namely the fact that t is a slice-contractible domain, the solvability of equation~\eqref{exp} follows almost immediately from
the complex holomorphic case. Nonetheless, differences with the complex case arise when looking to the case of a product domain and of a slice preserving function.


\begin{proposition}\label{one-slice-case} Let $\Omega$ be a slice-contractible domain and 
$g\in\So^{*}(\Omega)$. 
If $g$ is one-slice preserving, then there exists $f\in\So(\Omega)$ which preserves the same slice as $g$ and such that $\exp_{*}(f)=g$. Moreover, 

\noindent $\bullet$ if $\Omega$ is a slice domain, $g$ is slice preserving and
positive on $\Omega\cap\R$, then there exists a unique slice preserving $*$-logarithm of $g$;

\noindent $\bullet$  if $\Omega$ is a product domain and $g$ is slice preserving, 
then there exists a slice preserving $*$-logarithm of $g$.
\end{proposition}

\begin{proof}
Let us denote by $\C_{I}$ the preserved slice and consider the never-vanishing
restriction $g_{I}:\Omega_{I}\to\C_{I}$. Since either $\Omega_{I}$ is simply connected (whether $\Omega$ contains real points) or its two connected components are (whether $\Omega$ contains no real points),
then we can find a logarithm of $g_{I}$, that is a holomorphic function 
$f_{I}:\Omega_{I}\to \C_{I}$ such that $\exp(f_{I})=g_{I}$. Let $f$ be the
slice regular extension of $f_{I}$ to $\Omega$. Then we have 
$$
(\exp_{*}(f))_{I}=\exp(f_{I})=g_{I},
$$
and the Identity Principle entails the first part of the statement.

If $\Omega$ is a slice domain, $g$ is slice preserving and positive on $\Omega\cap\R$ then $g_{I}(\Omega_{I}\cap\R)\subset(0,+\infty)$ and therefore there exists a unique intrinsic logarithm $f_{I}:\Omega_{I}\to\C_{I}$ of $g_{I}$;
the extension of such a function to~$\Omega$ is the required slice preserving solution $f$. 

Now suppose that $\Omega$ is a product domain and that $g$ is slice preserving.
Fix any $I\in\SF$ and take $h=h_{0}+h_{1}I$ be a $\C_{I}$-preserving function such that $\exp_{*}(h)=g$. Formula~\eqref{munu} entails that 
$\exp_{*}(h_{0})\nu(h_{1}^{2})h_{1}I\equiv 0$, 
as $g$ is slice preserving. If $h_1\equiv0$, then setting $f=h=h_0$ gives the required function. Otherwise, $\nu(h_{1}^{2})$ must be identically zero and hence there exists $n\in\mathbb{N}\setminus\{0\}$ such that $h_{1}^{2}\equiv \pi^{2}n^{2}$ since $\nu(q)=0$ if and only if $q\in\R\setminus\{0\}$ and $q=\pi^{2}n^{2}$
for $n\in\mathbb{N}\setminus\{0\}$.
Thus $(h_{1}-\pi n)\cdot(h_{1}+\pi n)\equiv 0$ and Proposition~\ref{propGPS}
entails that either $h_{1}\equiv \pi n$ or $h_{1}=-\pi n$.
 

If $n$ is even, then  $\exp_{*}(\pm\pi nI)\equiv1$ on $\Omega$ and by taking $f=h_0$ 
we are done; if $n$ is odd, then  $\exp_{*}(\pm\pi n_{+}I)\equiv-1$ on $\Omega$ and thus we obtain the thesis by taking $f=h_0+\pi \mathcal J$.
%
%
\end{proof}

\begin{corollary}\label{cor1} Let $\Omega$ be a slice-contractible domain and $g$ a one-slice preserving function, then $g$ has a $*$-logarithm if and only if it is never vanishing.
In particular this holds for $g\in\So_{\R}(\Omega)$.
\end{corollary}

Trivially, if $\Omega$ is slice, 
$g$ is slice preserving and never-vanishing on $\Omega$ and it is negative on $\Omega \cap \R$, then there exists no slice preserving $*$-logarithm of $g$. Nonetheless, the family of $*$-logarithms of a slice preserving function with this features is quite large and displays an unexpected behaviour.

\begin{example}
For any $a\in\R\setminus\{0\}$ and $I\in\SF$ we have that $\exp_*(aI)\equiv \cos(a)+\sin(a)I$ . 
In particular $\exp_*(\pi i)=\exp_*(\pi j)\equiv -1$, while 
$$
\exp_*(\pi i+\pi j)=\exp_*\left(\sqrt2\pi\cdot \frac{i+j}{\sqrt2}\right)=\cos \left(\sqrt2\pi\right)+\sin \left(\sqrt2\pi\right)\frac{i+j}{\sqrt2}\neq 1=\exp_*(\pi i)*\exp_*(\pi j),
$$
giving an explicit example of application of Theorem 4.14 in~\cite{A-dF} (here $(\pi i)_v^s=(\pi j)_v^s=\pi^2$ and $2\langle \pi i , \pi j \rangle_*=0$, so that  $\exp_*(\pi i+\pi j)\not\equiv \exp_*(\pi i)*\exp_*(\pi j)$).
\end{example}

\begin{corollary}
Let $g\in\So^{*}(\Omega)$ be such that there exists $\chi\in\mathcal{R}^{*}(\Omega)$
for which $\chi^{-*}*g*\chi$ is one-slice preserving, then $g$ has a $*$-logarithm.
\end{corollary}

\begin{remark}\label{remark1}
Notice that if $g$ is conjugated to a one-slice preserving function via a never-vanishing
$\chi$, then $g_{v}^{s}$ has a square root. Nonetheless, the existence of  a square root of $g_{v}^{s}$ is not a sufficient condition for the existence of a $*$-logarithm of 
a never-vanishing function (see Example~\ref{non-existence}).
%
%
%
%
\end{remark}


Proposition~\ref{one-slice-case} allows us to prove a natural generalization to the quaternions of a classical result
in the theory of holomorphic functions that will be used later in the search for a solution of equation~\eqref{exp}.
The second part of the statement gives a uniqueness result obtained accordingly to the structure of the domain $\Omega$:
indeed, if the domain is slice, uniqueness up to a constant integer multiple of $2\pi$ holds, while in the case of
a product domain, uniqueness up to a constant integer  multiple of $2\pi$ holds for slice preserving functions only.

\begin{proposition}\label{cos-sin} Let $\Omega$ be a slice-contractible domain.
Given $a_{0},a_{1}\in\So_{\R}(\Omega)$ such that $a_{0}^{2}+a_{1}^{2}\equiv 1$, 
there exists $\gamma\in\So_{\R}(\Omega)$ such that
\begin{equation}\label{eq1}
\begin{cases}
\cos_{*}(\gamma)=a_{0}\\
\sin_{*}(\gamma)=a_{1}.
\end{cases}
\end{equation}
Moreover, 

\noindent$\bullet$ if $\Omega$ is a slice domain and $\tilde\gamma\in\So(\Omega)$ is another solution of~\eqref{eq1}, then there exists $k\in\mathbb{Z}$ such that $\tilde\gamma\equiv\gamma+2k\pi$;

\noindent$\bullet$ if $\Omega$ is a product domain and $\tilde\gamma\in\So_{\R}(\Omega)$ is another solution of~\eqref{eq1}, then there exists $k\in\mathbb{Z}$ such that $\tilde\gamma\equiv\gamma+2k\pi$.

\end{proposition}

\begin{proof}
Fix $I\in\mathbb{S}$ and consider the $\C_{I}$-slice preserving function given by $a=a_{0}+a_{1}I$.
Clearly $a^{s}\equiv 1$ so that $a$ is never vanishing. Proposition~\ref{one-slice-case} gives $f=f_{0}+f_{1}I\in\So_{I}(\Omega)$ such that 
$\exp_{*}(f)=a$; last equality can also be written as
\begin{equation}\label{chi}
\exp(f_{0})(\mu(f_{v}^{s})+\nu(f_{v}^{s})f_{v})=\exp(f_{0})(\mu(f_{1}^{2})+\nu(f_{1}^{2})f_{1}I)=a_{0}+a_{1}I.
\end{equation}
Since $a^{s}\equiv 1$, we obtain that $(\exp_{*}(f))^{s}\equiv\exp(2f_{0})\equiv 1$, that is $(\exp(f_{0}))^{2}\equiv 1$.
As $\So_{\R}(\Omega)$ is an integral domain, then either $\exp(f_{0})\equiv 1$ or $\exp(f_{0})\equiv -1$.
In the first case, the equalities
$$
\begin{cases}
\mu(f_{1}^{2})=\cos(f_{1}),\\
\nu(f_{1}^{2})f_{1}=\sin(f_{1}),
\end{cases}
$$
together with~\eqref{chi}, ensure that $\cos(f_{1})=a_{0}$ and $\sin(f_{1})=a_{1}$, so that $\gamma=f_{1}$ gives the required function.
In the second case, performing the same computations as above
gives that the function $\gamma =f_{1}+\pi$ is a solution of system~\eqref{eq1}.
%

Now suppose that $\Omega$ is slice and that $\tilde\gamma$ is another solution of System~\eqref{eq1}. The fact that $a_{0}$ is slice preserving
and $-1\leq a_{0}\leq 1$ on $\Omega\cap\R$, implies that 
$(\cos_{*}(\tilde\gamma))(t)=\cos(\tilde\gamma(t))$
belongs to the interval $[-1,1]$ for any $t\in\Omega\cap\R$ and therefore $\tilde\gamma(t)\in\R$ for any $t\in\Omega\cap\R$, showing that $\tilde\gamma$ is slice preserving.
Considering the functions on $\Omega\cap\R$,
trivially entails that $\tilde\gamma\equiv\gamma+2k\pi$, for some $k\in\Z$. 

We are left to consider the case in which $\Omega$ is a product domain and $\tilde\gamma$ is another slice preserving solution
of~\eqref{eq1}. Let $I\in\SF$, and consider the restrictions of $\gamma$ and $\tilde\gamma$ on $\Omega_{I}=\Omega_{I}^{+}\cup\Omega_{I}^{-}$. Trivially there exist $n_{+},n_{-}\in\Z$ such that $\tilde\gamma=\gamma+2\pi n_{+}$ on $\Omega_{I}^{+}$
and $\tilde\gamma=\gamma+2\pi n_{-}$ on $\Omega_{I}^{-}$.
By the Representation Formula we have, for $\beta>0$ and $J\in\SF$,
\begin{align*}
\tilde\gamma(\alpha+J\beta)&=\frac{1-JI}{2}\tilde\gamma(\alpha+I\beta)+\frac{1+JI}{2}\tilde\gamma(\alpha-I\beta)\\
&=\frac{1-JI}{2}(\gamma(\alpha+I\beta)+2\pi n_{+})+\frac{1+JI}{2}(\gamma(\alpha-I\beta)+2\pi n_{-})\\
&=\gamma(\alpha+J\beta) +\pi\left(n_{+}+n_{-}+J(n_{-}-n_{+})I\right)
\end{align*}
and this function is slice preserving if and only if $n_{-}-n_{+}= 0$. Thus $\tilde\gamma=\gamma+2\pi n_{+}$
 on $\Omega_{I}$ and the Identity Principle entails the proof.
\end{proof}

Notice that the hypothesis on the sliceness of the domain contained in the statement of Proposition~\ref{cos-sin} cannot be removed without adding a requirement on the behavior of the function $\tilde \gamma$, as it is shown in the following example.

\begin{example}
Let $\Omega=\HH\setminus\R$, fix any $I\in\SF$, and consider $\tilde\gamma:\HH\setminus\R\to\HH$ defined as 
$\tilde\gamma =2\pi \mathcal{J}I$. Clearly $\tilde\gamma_{0}\equiv 0$ and hence 
$\tilde\gamma^{*2}=-\tilde\gamma^{s}\equiv 4\pi^{2}$. Thus we have
$$\cos_{*}(\tilde\gamma)=\sum_{n\in\N}\frac{(-1)^{n}\tilde\gamma^{*2n}}{(2n)!}
=\sum_{n\in\N}\frac{(-1)^{n}( 4\pi^{2})^{n}}{(2n)!}=\sum_{n\in\N}\frac{(-1)^{n}(2\pi)^{2n}}{(2n)!}=\cos(2\pi)=1,
$$
and analogously, $\sin_{*}(\tilde\gamma)=0$. This gives a continuous family of functions parametrized by $I\in\SF$ defined on $\HH\setminus\R$ solving 
system~\eqref{eq1} with $a_{0}\equiv 1$ and $a_{1}\equiv 0$, in sharp contrast with the ``discrete'' uniqueness behavior which holds in the case of slice domains.
\end{example}

The above proposition allows us to give a more precise description of the  zero divisors whose ``real'' part is identically zero in the case when the product domain $\Omega$ is slice-contractible (see the third paragraph of~\cite{A-dFConcrete} for a comprehensive investigation of such functions). We denote by $\Sm(\Omega)$ the algebra of
slice semi-regular functions (see~\cite{GPSadvances,GPSdivision} for the definitions and a detailed study of the singularities and~\cite{A-dFSylvester} for an investigation in the flavor of a vector space structure over the field of slice preserving semi-regular functions).

\begin{proposition}\label{norealpart}
Let $\Omega$ be a slice-contractible product domain 
and let $f\in\Sm(\Omega)$ be a zero divisor such that $f_{0}\equiv 0$. Then there exists an orthonormal basis $(i,j,k)$ of 
$\mbox{Im}(\HH)$, $\alpha\in\Sm_{\R}(\Omega)$ and $\vartheta\in\So_{\R}(\Omega)$, such that
$$
f=\alpha i+\alpha\mathcal{J}\cos(\vartheta)j+\alpha\mathcal{J}\sin(\vartheta)k.
$$
\end{proposition}

\begin{proof}
Since $f=f_{v}$ is not identically zero, we can find an orthonormal basis $(i,j,k)$ of $\mbox{Im}(\HH)$,
such that $f=f_{1}i+f_{2}j+f_{3}k$, with $f_{1}\not\equiv 0$. 
Then, thanks to Proposition 2.14 in~\cite{A-dFSylvester}, we can write $f=-2(fi)_{0}i*\rho=2f_{1}i*\rho$, for a suitable idempotent 
$\rho=\frac{1}{2}+\rho_{1}i+\rho_{2}j+\rho_{3}k\in\So(\Omega)$. 
Now we have 
$$f=2f_{1}i*\rho=2f_{1}i*\left(\frac{1}{2}+\rho_{1}i+\rho_{2}j+\rho_{3}k\right)=f_{1}i-2f_{1}\rho_{1}+2f_{1}\rho_{2}k-2f_{1}\rho_{3}j,$$
and therefore, the fact that $f=f_{v}$ implies $\rho_{1}\equiv 0$, that is $\rho=\frac{1}{2}+\rho_{2}j+\rho_{3}k$.
As $\rho$ is an idempotent, we obtain $\rho_{2}^{2}+\rho_{3}^{2}\equiv-\frac{1}{4}$ and hence the functions $\varphi_{2}=2\mathcal{J}\rho_{2}$
and $\varphi_{3}=2\mathcal{J}\rho_{3}$ satisfy $\varphi_{2}^{2}+\varphi_{3}^{2}\equiv 1$.
%
%
%
%
Thus Proposition~\ref{cos-sin} implies we can find $\vartheta\in\So_{\R}(\Omega)$ such that $\varphi_{2}=\cos(\vartheta)$ and $\varphi_{3}=\sin(\vartheta)$. 
Setting $\alpha=f_{1}$, gives the required equality
$$f=\alpha i+\alpha\mathcal{J}\cos(\vartheta)j+\alpha\mathcal{J}\sin(\vartheta)k.$$
\end{proof}

\section{Uniqueness results}\label{uniqueness-results}

We now begin a detailed investigation on the possible family of solutions of equation~\eqref{exp}
starting from the case of slice preserving functions. Our first statement classifies slice preserving functions
which share the same exponential, giving a first uniqueness result for the exponential problem in the case of slice domains.

\begin{proposition}\label{prophvequivkv0}
Let $h_{0},\tilde{h}_{0}\in\So_{\R}(\Omega)$ be such that 
\begin{equation}\label{hvequivkv0}
\exp_{*}(h_{0})\equiv\exp_{*}(\tilde{h}_{0}).
\end{equation}
\begin{itemize}
\item If $\Omega$ is a slice domain then $h_{0}\equiv \tilde{h}_{0}$. 
\item If $\Omega$ is a product domain, then there exists $n\in\Z$
such that  $h_{0}=\tilde{h}_{0}+2\pi n\mathcal{J}$.
\end{itemize}
\end{proposition}

\begin{proof}
Since both $h_{0}$ and $\tilde{h}_{0}$ are slice preserving, Proposition 4.3 in~\cite{A-dF} gives us the possibility to work on the difference $f_{0}=h_{0}-\tilde{h}_{0}$ which is a solution of 
\begin{equation}\label{equazioneexp=1}
\exp(f_{0})\equiv 1.
\end{equation}

If $\Omega$ is a slice domain, equality~\eqref{equazioneexp=1} gives $f\equiv 0$ on $\Omega\cap\R$. The identity principle entails $f_{0}\equiv 0$ on $\Omega$, that is $h_{0}\equiv \tilde{h}_{0}$ on $\Omega$ and hence the thesis. 

Assume now that $\Omega\cap\R=\emptyset$. Fix $I\in\SF$ and restrict equality~\eqref{equazioneexp=1} to $\Omega\cap\C_{I}$.
Since $f_{0}$ is slice preserving and $\Omega_{I}$ has two connected components, $\exp(f_{0})\equiv 1$ implies 
the existence of $n_{+},n_{-}\in\Z$ such that 
$$
f_{0}(x+Iy)=\begin{cases}
2\pi n_{+}I,\quad \mbox{for }y>0,\\
2\pi n_{-}I,\quad \mbox{for }y<0.
\end{cases}
$$
For $y>0$ and $J\in\SF$, the Representation Formula yields
\begin{align*}
f_{0}(x+Jy)&=\frac{1-JI}{2}f_{0}(x+Iy)+\frac{1+IJ}{2}f_{0}(x-Iy)\\
&=(1-JI)\pi n_{+}I+(1+I)\pi n_{-}I= \pi((n_{+}+n_{-})I+(n_{+}-n_{-})J).
\end{align*}
Since $f_{0}\in\So_{\R}(\Omega)$, this equality gives $n_{+}+n_{-}=0$ and thus $f_{0}(x+Jy)=2\pi n_{+}J$ for any $y>0$ and $J\in\SF$, that is $h_{0}=\tilde{h}_{0}+2\pi n_{+}\mathcal{J}$.
\end{proof}

\begin{corollary}\label{cor2}
Let $h,\tilde{h}\in\So(\Omega)$ be such that 
\begin{equation}\label{hvequivkv}
\exp_{*}(h)\equiv\exp_{*}(\tilde{h}).
\end{equation}
\noindent $\bullet$ If $\Omega$ is a slice domain, then $h_{0}\equiv \tilde{h}_{0}$ and $\exp_{*}(h_{v})\equiv \exp_{*}(\tilde{h}_{v})$. 

\noindent $\bullet$ If $\Omega$ is a product domain, then there exists $n\in\Z$ such that
$h_{0}=\tilde{h}_{0}+\pi n\mathcal{J}$. In this case 
$\exp_{*}(h_{v})\equiv \exp_{*}(\tilde{h}_{v})$ if $n$ is even and $\exp_{*}(h_{v})\equiv -\exp_{*}(\tilde{h}_{v})$ if $n$ is odd.
\end{corollary}

\begin{proof}
Using~\cite[formula (4.4)]{A-dF}, equality~\eqref{hvequivkv} implies $\exp(2h_{0})=\exp(2\tilde{h}_{0})$. As $h_{0}, \tilde{h}_{0}$ are
slice preserving functions, Proposition~\ref{prophvequivkv0} gives that either $h_{0}\equiv \tilde{h}_{0}$ or $\Omega$ is a product domain and $2h_{0}= 2\tilde{h}_{0}+2\pi n\mathcal{J}$ for a suitable $n\in\Z$. A straightforward computation yields the thesis.
\end{proof}


We now study when two functions $h,\tilde{h}$ give the same $*$-exponential. The first case we analyze is when $\exp_{*}(h)$ is slice preserving. In order to simplify notations, we set
\begin{equation}\label{nuiszero}
\mathfrak{N}(\Omega):=\{f\in\So(\Omega)\,|\,\exists\, m\in\Z\setminus\{0\}, f_{v}^{s}=m^{2}\pi^{2}\}\cup\So_{\R}(\Omega).
\end{equation}

We strongly underline that if $f\in\mathfrak{N}(\Omega)$, then either $f_{v}\equiv 0$ or $f_{v}$ is never vanishing; in both cases
$f_{v}$ has no non-real isolated zeroes and $f_{v}^{s}$ always admits a square root (which, by the way, is constant).

\begin{theorem}
Let $h,\tilde{h}\in\So(\Omega)$ be such that $\exp_{*}(h)=\exp_{*}(\tilde{h})\in\So_{\R}(\Omega)$.

\noindent $\bullet$ If $\Omega$ is a slice domain, then $h_{0}\equiv \tilde{h}_{0}$, $h,\tilde{h}\in\mathfrak{N}(\Omega)$ and $\sqrt{h_{v}^{s}}\equiv \sqrt{\tilde{h}_{v}^{s}}$ \textnormal{(mod $2\pi$)}.

\noindent $\bullet$ If $\Omega$ is a product domain,  then there exists $n\in\Z$ such that
$h_{0}=\tilde{h}_{0}+\pi n\mathcal{J}$. Moreover
$h,\tilde{h}\in\mathfrak{N}(\Omega)$ and $\sqrt{h_{v}^{s}}\equiv \sqrt{\tilde{h}_{v}^{s}}+n\pi$ \textnormal{(mod $2\pi$)}.

\end{theorem}

\begin{proof}

By Corollary~\ref{cor2} we have that either $h_{0}= \tilde{h}_{0}$ or 
$\Omega$ is a product domain and there exists $n\in\Z$ such that
$h_{0}=\tilde{h}_{0}+\pi n\mathcal{J}$.
In particular these equalities imply 
\begin{align}
\exp(h_{0})&\equiv\exp(\tilde{h}_{0}),\qquad\quad\mbox{for }n\mbox{ even}\label{i}\\
\exp(h_{0})&\equiv-\exp(\tilde{h}_{0}),\qquad\mbox{for }n\mbox{ odd}\label{ii}
\end{align}
where the second one can occur only if $\Omega$ is a product domain.
We start by considering the case of a slice domain. In this case we have $\exp_{*}(h_{v})=\exp_{*}(\tilde h_{v})\in\So_{\R}(\Omega)$. 
If $h_{v}=\tilde h_{v}\equiv 0$, we are done. Otherwise, we can suppose $h_{v}\not\equiv 0$, which implies
$\nu(h_{v}^{s})\equiv 0$ and therefore guarantees that there exists $m\in\Z\setminus\{0\}$ such that 
$h_{v}^{s}=m^{2}\pi^{2}$. As $\nu(h_{v}^{s})h_{v}=\nu(\tilde{h}_{v}^{s})\tilde h_{v}\equiv 0$, we obtain that
either $\tilde h_{v}\equiv 0$ or $\tilde h_{v}\not\equiv 0$ (and hence $\nu(\tilde h_{v}^{s})\equiv 0$).
In the first case, we have $\exp_{*}(h_{v})=\mu(h_{v}^{s})\equiv 1$ thus showing that $m$ is even
and that $\sqrt{h_{v}^{s}}\equiv \sqrt{\tilde h_{v}^{s}}$ \textnormal{(mod $2\pi$)}.
In the second one, we can find $n\in\Z\setminus\{0\}$, such that $\tilde h_{v}^{s}=n^{2}\pi^{2}$; as $\mu(h_{v}^{s})=\mu(\tilde h_{v}^{s})$, we find that $n$ and $m$ have the same parity, thus showing again $\sqrt{h_{v}^{s}}\equiv \sqrt{\tilde{h}_{v}^{s}}$ \textnormal{(mod $2\pi$)}.

The case of a product domain is obtained following the same lines of reasoning, by studying separately the two cases given by formulas~\eqref{i}
and~\eqref{ii}.
\end{proof}

We now turn to the case when $\exp_{*}(h)$ is not slice preserving;
under this hypothesis, if $\Omega$ is slice we find a dichotomy: either $h_{v}$ and $\tilde h_v$ have no non-real isolated zeroes and are such that $h_v^s$ and $\tilde h_v^s$ have a square root (independently from the fact that $\Omega$ is slice-contractible), in which case we find a ``discrete'' family of functions producing the same $*$-exponential, or $h=\tilde h$, that is,
there is an unexpected uniqueness result. An analogous, more refined statement can be
obtained also in the case of a product domain.



\begin{theorem}\label{unicity}
Let $h,\tilde{h}\in\So(\Omega)$ be such that $h\neq\tilde h$ and $\exp_{*}(h)=\exp_{*}(\tilde{h})\not\in\So_{\R}(\Omega)$.

\noindent $\bullet$ If $\Omega$ is a slice domain, then 
$h_{0}\equiv \tilde{h}_{0}$, 
both $h_{v}$ and $\tilde h_{v}$ 
have no non-real isolated zeroes, both $h_{v}^{s}$ and $\tilde{h}_{v}^{s}$ have a square root on $\Omega$ and there exist $m\in\Z\setminus\{0\}$, $\alpha\in\So_\R(\Omega)$ and $H_v\in\So(\Omega)$ with $H_v^s\equiv1$  such that $h_v=\alpha H_v$ and 
$\tilde{h}_{v}=(\alpha +2\pi m)H_v=h_{v}+2\pi m H_v$, so that
$$
\tilde h=h+2\pi m H_v.
$$
\noindent $\bullet$ If $\Omega$ is a product domain,  then one of the following holds
\begin{enumerate}
\item there exists $n\in\Z\setminus\{0\}$ such that
$$\tilde h=h+2\pi n\mathcal{J};$$
\item both $h_v$ and $\tilde h_v$ are not zero divisors and have no non-real isolated zeroes, both $h_{v}^{s}$ and $\tilde{h}_{v}^{s}$ have a square root on $\Omega$ and there exist
$n,m\in\mathbb Z$ such that $n\equiv m$ \textnormal{(mod $2$)} and $m\neq0$, $\alpha\in\So_\R(\Omega)$ and $H_v\in\So(\Omega)$ with $H_v^s\equiv1$, such that 
$h_v=\alpha H_v$ and 
$$\tilde{h}=h_{0}+\pi n\mathcal{J}+(\alpha +\pi m)H_v=h+\pi (n\mathcal{J}+ m H_v).$$
\end{enumerate}
\end{theorem}

\begin{proof}
If $h_v=\tilde h_v$, then $\exp(h_0)=\exp(\tilde h_0)$, so the hypothesis $h\neq \tilde h$ and  Proposition~\ref{prophvequivkv0} give that $\Omega$ is a product domain and  there exists $n\in\Z\setminus\{0\}$ such that
$h_{0}=\tilde{h}_{0}+2\pi n\mathcal{J}$.
Thus, from now on, we assume that $h_v\neq \tilde h_v$.

If $\Omega$ is a slice domain, Corollary~\ref{cor2} gives $h_{0}=\tilde h_{0}$ and
$\exp_{*}(h_{v})=\exp_{*}(\tilde{h}_{v})\not\in\So_{\R}(\Omega)$. 
We have
$\nu(h_{v}^{s})h_{v} \equiv \nu(\tilde{h}_{v}^{s})\tilde{h}_{v}\not\equiv 0$. This implies that $h_{v}$ and $\tilde{h}_{v}$
are linearly dependent on $\So_{\R}(\Omega)$ and thus commute (see~\cite{A-dF}, Proposition 2.10). 
Now choose $p\in\Omega\setminus\R$ such that $\nu(h_{v}^{s})h_{v}$ is never vanishing on $\SF_{p}$ and denote by $\widetilde{\Omega}$ a slice-contractible product domain contained in $\Omega$ such that $\nu(h_{v}^{s})h_{v}$ is never vanishing on $\widetilde{\Omega}$. 

In particular this means that $h_v^s$ is never-vanishing on $\widetilde \Omega$ and therefore there exists a square root $\alpha$ of $h_v^s$ which is never vanishing on $\widetilde \Omega$.
Moreover, the equality $\nu(h_{v}^{s})h_{v} \equiv \nu(\tilde{h}_{v}^{s})\tilde{h}_{v}$ gives that
$\nu(\tilde{h}_{v}^{s})$ is never vanishing on $\widetilde{\Omega}$ and thus
there exists $\beta\in\So_{\R}(\widetilde{\Omega})$ such that $\tilde h_{v}=\beta h_{v}$. As $h\not\equiv \tilde h$
on $\Omega$, the Identity principle gives that $\beta$ is not identically equal to $1$ on $\widetilde{\Omega}$.
Since $h_{v}$ and $\tilde{h}_{v}$ commute, we can apply~\cite[Proposition 4.3]{A-dF} and we
are left to study $\exp_{*}((1-\beta)h_{v})\equiv 1$. As $h_{v}\not\equiv 0$, this implies  that there exists $m\in\Z\setminus\{0\}$ 
such that $((1-\beta)h_{v})^{s}=4m^{2}\pi^{2}$, that is $(1-\beta)^{2}h_{v}^{s}=4m^{2}\pi^{2}$ and
therefore gives $\beta=\frac{2\pi m}{\alpha}+1$, up to a possible change of sign of $m$.
We then obtain 
\begin{equation}\label{eqsqrt0}
\tilde h_{v}=\left(\frac{2\pi m}{\alpha}+1\right)h_{v},
\end{equation}
on $\widetilde{\Omega}$, which can also be written as
\begin{equation}\label{eqsqrt}
\alpha\cdot\tilde h_{v}=\left(2\pi m+\alpha\right)h_{v}.
\end{equation}
By computing the symmetrized function of both members of equation~\eqref{eqsqrt}, we obtain the following equality
$h_{v}^{s}\tilde h_{v}^{s}=\left(2\pi m+\alpha\right)^{2}h_{v}^{s}$.
As $\tilde h_{v}^{s}$ is never vanishing on $\widetilde{\Omega}$, we also have 
$\tilde h_{v}^{s}=\left(2\pi m+\alpha\right)^{2}$,
from which we infer 
$$
\alpha=\frac{\tilde h_{v}^{s}-h_{v}^{s}-4\pi^{2} m^{2}}{4\pi m},
$$
on $\widetilde{\Omega}$, thanks to $\alpha^2=h_v^s$. By squaring both members, we thus get 
$h_{v}^{s}=\frac{1}{16\pi^{2}m^{2}}\left(\tilde h_{v}^{s}-h_{v}^{s}-4\pi^{2} m^{2}\right)^{2}$.
Last equality was obtained on $\widetilde{\Omega}$, but since both members are defined and regular on $\Omega$, by the Identity Principle
we have that it holds on the whole $\Omega$, thus showing that $h_{v}^{s}$ has a square root in $\Omega$. Up to a change of sign, we can suppose that it agrees with the previous one on $\widetilde{\Omega}$, so we still denote it by $\alpha\in\So_\R(\Omega)$.
A further application of the Identity Principle shows that Equality~\eqref{eqsqrt} continues on $\Omega$.

Now suppose that $q_{0}$ is a non real isolated zero of $h_{v}$; thus $h_{v}^{s}$ and therefore $\alpha$ are identically
zero on $\SF_{q_{0}}$. The left hand side of equality~\eqref{eqsqrt} is then identically zero on $\SF_{q_{0}}$, while
the right hand side is equal to $2\pi m\cdot h_{v}$; as $m\neq 0$, last function has an isolated zero in $q_0$. This contradiction shows that $h_{v}$ cannot have non real isolated zeroes. 

Thanks to Proposition~\ref{quotient}, the function $H_v:=\frac{h_{v}}{\alpha}$ is a well defined  slice-regular function 
on $\Omega$ with $H_v^s\equiv 1$; as $h_v=\alpha H_v$,  by the zero-product property,
 Equality~\ref{eqsqrt} can also be written in the form $\tilde h_{v}=
 \left(2\pi m+\alpha\right)H_v$, which holds on the whole of~$\Omega$ by a further application of the Identity Principle.

Now, since $\tilde h_v$ differ for the slice-preserving factor  $\left(2\pi m+\alpha\right)$ only from $H_v$, which is never-vanishing, we have that the only zeroes of $\tilde h_v$ are the zeroes of  $\left(2\pi m+\alpha\right)$ and therefore $\tilde h_v$ has no non-real isolated zeroes as well. Lastly, $\tilde h_v^s= \left(2\pi m+\alpha\right)^2$, which ensures that 
$\tilde h_v^s$ has a square root on $\Omega$.

Now we turn our attention to the case in which $\Omega$ is a product domain. Corollary~\ref{cor2} gives that
there exists $n\in\Z$ such that
$h_{0}=\tilde{h}_{0}+\pi n\mathcal{J}$ and moreover, either
$\exp_{*}(h_{v})\equiv \exp_{*}(\tilde{h}_{v})$ if $n$ is even or $\exp_{*}(h_{v})\equiv -\exp_{*}(\tilde{h}_{v})$ if $n$ is odd.
We first deal with the case when $n$ is even; in particular this gives $\nu(h_{v}^{s})h_{v} \equiv \nu(\tilde{h}_{v}^{s})\tilde{h}_{v}\not\equiv 0$.

We first look at the case in which $h_{v}$ is a zero divisor.
As $h_{v}^{s}\equiv 0$, we have $\nu(h_{v}^{s})\equiv 1$,
thus the above equality becomes $h_{v}=\nu(\tilde h_{v}^{s})\tilde h_{v}\not\equiv 0$, which in particular implies that
$\nu(\tilde h_{v}^{s})\not\equiv 0$. By taking the symmetrized function of both members of the above equality,
we obtain $0\equiv h_{v}^{s}=\nu(\tilde h_{v}^{s})^{2}\tilde h_{v}^{s}$, showing that $\tilde h_{v}^{s}\equiv 0$, too.
Hence $\nu(\tilde h_{v}^{s})\equiv 1$ and therefore $h_{v}\equiv \tilde h_{v}$ which is a contradiction to the assumption $h_v\neq \tilde h_v$.

We now turn to the case in which $h_{v}$ is not a zero divisor. As $n$ is even,  $\nu(h_{v}^{s})h_{v} \equiv \nu(\tilde{h}_{v}^{s})\tilde{h}_{v}$ is not a zero divisor and $h_v\neq \tilde h_v$, we can argue as in the case of a slice domain obtaining that both 
$h_{v}$ and $\tilde h_{v}$ have no non-real isolated zeroes and there exist an even $m\in\Z\setminus\{0\}$, $\alpha\in\So_\R(\Omega)$ and $H_v\in\So(\Omega)$ with $H_v^s\equiv1$  such that 
$h_v=\alpha H_v$ and $\tilde{h}_{v}=(\alpha +\pi m)H_v=h_{v}+\pi m H_v$.

Finally, we consider the case in which $n$ is odd. The only difference with the above reasoning is due to the fact that 
$\exp_{*}(h_{v})\equiv -\exp_{*}(\tilde{h}_{v})$. In the case when $h_{v}$ is a zero divisor, again we get that $\tilde h_v$ is a zero-divisor. Thus $\nu(\tilde h_v^s)\equiv1$ and we obtain $1+h_{v}=-(1+\tilde h_{v})$ which is equivalent to $h_{v}+\tilde h_{v}=-2$. This is a contradiction, because $-2$ is slice preserving and different from $0$, while $h_v-\tilde h_v$ coincides with its vector part.
Finally, in the case when $h_{v}$ is not a zero divisor, the above reasoning gives that also $\tilde h_{v}$ is not a zero divisor, moreover, 
both $h_{v}$ and $\tilde h_{v}$ have no non-real isolated zeroes and there exists $m\in\Z$ odd, $\alpha\in\So_\R(\Omega)$, $H_v\in\So(\Omega)$ with $H_v^s\equiv 1$ such that $h_v=\alpha H_v$ and 
$\tilde{h}_{v}=(\alpha +\pi m)H_v=h_{v}+\pi mH_v$.
\end{proof}

\begin{remark}\label{remarkHV}
Notice that the function $H_{v}\in\So(\Omega)$ such that $H_{v}^{s}\equiv 1$
which appears in the previous statement is unique up to a change of sign and
can be interpreted as the quotient of $h_{v}$ by a square root of
$h_{v}^{s}$. Indeed, if $\alpha H_{v}=\beta L_{v}=h_{v}$ for
$\alpha,\beta\in\So_{\R}(\Omega)$ and $H_{v}^{s}=L_{v}^{s}\equiv 1$ we have
$\alpha^{2}=\beta^{2}$, thus either $\alpha=\beta$ (which gives $L_{v}=H_{v}$)
or $\alpha=-\beta$ (which gives $L_{v}=-H_{v}$). Moreover, $h_{v}^{s}=(\alpha H_{v})^{s}=\alpha^{2}H_{v}^{s}=\alpha^{2}$, so that $\alpha$ is a square root
of $h_{v}^{s}$.
\end{remark}

To stress the relevance, and also the unexpectedness, of the above theorem, we give a couple of partial restatement which underline the uniqueness 
result when $\Omega$ is slice and the ``vector'' part of the function has an isolated zero.

\begin{corollary}
Let $\Omega$ be slice
and $h,\tilde h\in\So(\Omega)$ be such that $\exp_{*}(h)\not\in\So_{\R}(\Omega)$.
If $h_{v}$ has a non real isolated zero, then $\exp_{*}(h)=\exp_{*}(\tilde h)$ if and only if $h\equiv\tilde h$.
\end{corollary}
\begin{corollary}
Let $\Omega$ be slice
and $h,\tilde h\in\So(\Omega)$ be such that $\exp_{*}(h)=\exp_{*}(\tilde h)\not\in\So_{\R}(\Omega)$.
If there exists $q\in\Omega$ such that $h(q)=\tilde h(q)$, then $h\equiv \tilde h$.
\end{corollary}

\begin{proof}
As $\Omega$ is slice, if $h\not\equiv\tilde h$, then there exist $m\in\mathbb{Z}\setminus \{0\}$ and $H_{v}$ with $H_{v}^{s}\equiv 1$ such that $\tilde h=h+2\pi m H_{v}$. Thus $\tilde h(q)=h(q)+2\pi m H_{v}(q)$ gives $mH_{v}(q)=0$. Since $H_{v}^{s}\equiv 1$, this entails $m=0$, that
is a contradiction.
\end{proof}

\section{Existence results for the $*$-logarithm}\label{the-core}

As a first consequence of the results obtained in Section~\ref{first_existence_results}, Proposition~\ref{one-slice-case} allows us to restrict our attention to a particular class of never vanishing functions.

\begin{remark}\label{r1} 
Let $\Omega$ be a slice-contractible domain.
For any $g\in\So^{*}(\Omega)$, we have that $g^{s}$ belongs to $\So^{*}(\Omega)\cap\So_{\R}(\Omega)$ and it is positive on the reals if $\Omega\cap \R\neq \emptyset$. 
Then Proposition~\ref{one-slice-case} entails the existence of a 
$\psi_{g}\in\So_{\R}(\Omega)$ such that $\exp(\psi_{g})=g^{s}$
(and $\psi_{g}$ is unique is $\Omega$ is slice). 
A trivial computation shows that $\exp(-\psi_{g}/2)g$ belongs to $\So^{1}(\Omega)$. Moreover, since $\exp(-\psi_{g}/2)$ is slice preserving, by~\cite[Corollary 4.4]{A-dF}, we
have that $g$ has a $*$-logarithm if and only if $\exp(-\psi_{g}/2)g$ has.
Thanks to these considerations, without loss of generality, we can 
reduce ourselves to study equation~\eqref{exp} in the case when $g^{s}\equiv 1$. 
\end{remark}

\begin{assumption}\label{notslicepreserving}
Since Corollary~\ref{cor1} gives the existence of a $*$-logarithm for all never vanishing slice preserving regular functions, from now on
we consider equation~\eqref{exp} only in the case when $g\not\in\So_{\R}(\Omega)$, that is $g_{v}\not\equiv 0$. 
\end{assumption}

\begin{proposition}\label{propos1}
Let $g\in\So^{1}(\Omega)$ and suppose that $f$ is a $*$-logarithm of $g$. If $\Omega$ is slice, then $f_{0}\equiv 0$;
if $\Omega$ is product, then there exists $n\in\mathbb{Z}$ such that
$f_{0}\equiv n\pi\mathcal{J}$. In particular,
\begin{itemize}
\item  if $\Omega$ is slice, then
any $*$-logarithm of $g$ has ``real part'' identically zero;
\item if $\Omega$ is product, then, up to substituting $g$ with $-g$, we can find a $*$-logarithm of $g$
whose real part is identically zero.
\end{itemize}
\end{proposition}

\begin{proof}
Let us assume that $f$ is a solution of Equation~\eqref{exp}. Thanks to~\cite[Proposition 4.13]{A-dF}, as $g^{s}\equiv 1$, we have that $\exp(2f_{0})\equiv 1\equiv \exp(0)$. Thus, corollary~\ref{cor2} ensures that $f_{0}\equiv 0$
if $\Omega$ is a slice domain, while there exists $n\in\mathbb{Z}$
such that $f_{0}\equiv n\pi\mathcal{J}$ if $\Omega$ is a product domain.
\end{proof}

The above proposition tells us that we can limit ourselves to look for solutions of $\exp_{*}(f_{v})=g$
if $\Omega$ is slice or $\Omega$ is product and $n$ is even and $\exp_{*}(f_{v})=- g$ if $\Omega$ is product and $n$ is odd.

%
%

The following result sets the existence of a $*$-logarithm for a never-vanishing function whose ``vector part" is a zero-divisor (obviuosly, this case can occur only if $\Omega$ is a product domain).
\begin{proposition}\label{zero-divisor}
Let $\Omega$ be a slice contractible domain and $g\in\So^*(\Omega)$ be such that $g_v$ is a zero-divisor. Then there exists $f\in\So(\Omega)$ such that  $\exp_*(f)=g$.
\end{proposition}

\begin{proof}
By Remark~\ref{r1}, we can suppose that $g^s=g_0^2+g_v^s\equiv 1$. As $g_v$ is a zero-divisor, we have $g_v^s\equiv 0$ and therefore $g_0^2\equiv1$, which entails that either $g_0\equiv1$ or $g_0\equiv-1$. 
In the first case a trivial computation gives
$\exp_*(g_v)=\mu(g_v^s)+\nu(g_v^s)g_v=\mu(0)+\nu(0)g_v=1+g_v=g_0+g_v=g$, in the second
$\exp_*(\pi \mathcal J-g_v)=\exp(\pi \mathcal J)*(\mu(g_v^s)+\nu(g_v^s)(-g_v))=-(\mu(0)-\nu(0)g_v)=-1+g_v=g_0+g_v=g$.
\end{proof}

\begin{assumption}\label{assumptionzerodiv}
Thanks to Proposition~\ref{zero-divisor} we can refine~\ref{notslicepreserving} by assuming that $g_{v}$ is neither identically zero nor a zero divisor.
\end{assumption}


Thanks to formula~\eqref{munu}, if $g=\exp_{*}(f_{v})$, then 
\begin{equation}\label{system1}
\begin{cases}
\mu(f_{v}^{s})=g_{0},\\
\nu(f_{v}^{s})f_{v}=g_{v}.
\end{cases}
\end{equation}

A first simple necessary condition in order to ensure the solvability of equation~\eqref{exp} entails the
behavior of $g$ at non-real isolated zeroes of $g_{v}$. As a surprising consequence we obtain that, the presence of
non-real isolated zeroes of $g_{v}$ could be an obstruction to the existence
of a $*$-logarithm of $g$.
This feature underlines the strong difference between the complex and the quaternionic case for the exponential function.


\begin{proposition}\label{g0=1}
If $g\in\So^{1}(\Omega)$ has a $*$-logarithm, we have that
\begin{enumerate}
\item if $\Omega$ is a slice domain and $q_{0}$ is a non-real isolated zero of $g_{v}$, then $g(q_{0})=1$;
\item if $\Omega$ is a product domain and $q_{0}, q_{1}$ are non-real
isolated zeroes of $g_{v}$, then, either $g(q_{0})=g(q_{1})=1$
or $g(q_{0})=g(q_{1})=-1$.
\end{enumerate}
%
%
%
%
\end{proposition}

\begin{proof}

Let $\Omega$ be a slice domain and $f$ be a $*$-logarithm of $g$.
By Proposition~\ref{propos1} we have that $f=f_{v}$.
Then, if $q_{0}$ is a non-real isolated zero of $g_{v}$, then the second equation of system~\eqref{system1},
implies that $\nu(f_{v}^{s})f_{v}$ has a non-real isolated zero at $q_{0}$. Since $\nu(f_{v}^{s})$ is slice preserving,
then $q_{0}$ is a non-real isolated zero of $f_{v}$ and thus $f_{v}^{s}(q_{0})= 0$. Therefore the first equation
gives $g(q_{0})=g_{0}(q_{0})=\mu(0)=1$.

If $\Omega$ is a product domain, let $f=f_{v}$ be a $*$-logarithm of either $g$ or $-g$. Again, $\nu(f_{v}^{s})f_{v}$ has non-real isolated zeroes at $q_{0}$ and $q_{1}$, so that $f_{v}$ has non-real isolated zeroes at $q_{0}$ and $q_{1}$
and thus $f_{v}^{s}(q_{0})=f_{v}^{s}(q_{1})=0$. The first equation of System~\eqref{system1} thus gives $g(q_{0})=g(q_{1})=1$, if $f$ is a $*$-logarithm of $g$, and $-g(q_{0})=-g(q_{1})=1$, if $f$ is a $*$-logarithm of $-g$.
\end{proof}


In particular, if $\Omega$ is a slice domain, the previous proposition gives a
strong obstruction to the existence of a $*$-logarithm. The following two
corollaries give explicit restraints to the existence of a $*$-logarithm: the first one
applies to any slice domain, while the function must have a special form, the second
one holds in a smaller class of domains, but for a larger class of functions.

\begin{corollary}\label{sign-obstruction}
Let $\Omega$ be a slice domain and $f\in\So(\Omega)$ be such that $f_{v}$ has a non-real isolated zero. Then $-\exp_{*}(f)$ has no $*$-logarithm. 
\end{corollary}

\begin{proof}
As $-\exp_{*}(f)=-\exp(f_{0})\exp_{*}(f_{v})$, then $-\exp_{*}(f)$ has a $*$-logarithm
if and only if $-\exp_{*}(f_{v})$ has. Now notice that $(-\exp_{*}(f_{v}))^{s}\equiv 1$,
so, if $-\exp_{*}(f_{v})$ has a $*$-logarithm, then it fulfills the hypotheses of Proposition~\ref{g0=1}.
Let us denote by $q_{0}$ a non-real isolated zero of $f_{v}$. On
$\SF_{q_{0}}$, we have $f_{v}^{s}\equiv 0$, hence $\mu(f_{v}^{s})=1$ on $\SF_{q_{0}}$ and thus $-\exp_{*}(f_{v})(q_{0})=-1$ which is a contradiction to Proposition~\ref{g0=1}.
%
%
\end{proof}

\begin{corollary}
Let $\Omega$ be a slice contractible slice domain and $g\in\So(\Omega)$ be a never-vanishing function such that $g_{v}$ has a non-real isolated zero. Then at least one 
between $g$ or $-g$ has no $*$-logarithm. 
\end{corollary}

\begin{proof}
By using the notation
contained in Remark~\ref{r1}, we have that $\psi_{g}=\psi_{-g}$, thus $\exp_{*}(-\psi_{g}/2)g$
and $\exp_{*}(-\psi_{-g}/2)(-g)$ are opposite one another. 
Proposition~\ref{g0=1} ensures that one of these two functions has no $*$-logarithm and thus at least one of the two functions $g$ and $-g$ has no $*$-logarithm, too.
\end{proof}


\begin{example}\label{non-existence}
In view of Corollary~\ref{sign-obstruction}, we may give a large family of examples 
of never-vanishing functions without $*$-logarithm. Take for instance
the polynomials $(q-i)*j$, $(q-i)^{*2}*j$ or $(q-i)*(q-2j)*(-2i+j)$. It is not difficult to check that
these three polynomials
have no ``real part'' and have only non-real isolated zeros. Therefore,
the functions $-\exp_{*}((q-i)*j)$, $-\exp_{*}((q-i)^{*2}*j)$ and $-\exp_{*}((q-i)*(q-2j)*(-2i+j))$
have no $*$-logarithm. In particular, notice that $-\exp_{*}((q-i)^{*2}*j)_{v}=-\nu((q^{2}+1)^{2})(q-i)^{*2}*j$ and its symmetrized function is given by
$\nu((q^{2}+1)^{2})^{2}(q^{2}+1)^{2}$ which trivially has the square root 
$\nu((q^{2}+1)^{2})(q^{2}+1)$. This provides the example 
of a function $g\in\So^{1}(\HH)$ such that $g_{v}^{s}$ has a square root but which has no $*$-logarithm
we were referring in Remark~\ref{remark1}.
\end{example}

\begin{remark}\label{oneslice}
Notice that if $g$ is one-slice preserving, then $g_{v}$ cannot have non-real isolated zeroes. Indeed, if the preserved slice
lies is $\C_{I}$, then $g_{v}=g_{1}I$, where $g_{1}\in\So_{\R}(\Omega)$ and hence $g_{v}$ has only real and
spherical isolated zeroes.
\end{remark}

If $\Omega$ is slice contractible, as $g_{v}$ is neither identically zero nor a zero divisor by Assumption~\ref{assumptionzerodiv},
following the outline of the proof of~\cite[Proposition 3.1]{A-dFAMPA} (and taking as $\Omega_{I}^{+}$ the unitary disc centered at $2i$ and $h$ without spherical zeroes in the case of a product domain), we can find $\alpha\in\So_{\R}(\Omega)$ and $W\in\So(\Omega)$, such that 
\begin{equation}\label{gvalphaW}
g_{v}=\alpha W,
\end{equation}
where $W\not \equiv 0$ is not a zero divisor and has neither real nor spherical zeroes.
In particular, we notice that $W_{0}\equiv 0$.
The second equation of system~\eqref{system1} becomes now 
\begin{equation}\label{alphaW}
\nu(f_{v}^{s})f_{v}=\alpha W.
\end{equation}

\begin{lemma}\label{betaW}
Let $f$ be a slice regular function satisfying equation~\eqref{alphaW}, where $\alpha\in\So_{\R}(\Omega)\setminus\{0\}$ and $W$ has
neither real nor spherical zeroes. Then there exists $\beta\in\So_{\R}(\Omega)$, such that $f_{v}=\beta W$.
\end{lemma}
\begin{proof}
Equation~\eqref{alphaW} ensures that $\nu(f_{v}^{s})$ is not identically zero.
If $q_{0}$ is a real isolated zero of $\nu(f_{v}^{s})$ of multiplicity $n$, then it is a real isolated zero of $\alpha W$ of
multiplicity greater or equal than $n$; as $W$ has no real zeroes, then $q_{0}$ is a real isolated zero of $\alpha$ with
multiplicity greater or equal than $n$. The same holds for spherical zeroes of $\nu(f_{v}^{s})$. Chose $I\in\mathbb{S}$
and consider the restriction of both $\nu(f_{v}^{s})$ and $\alpha$ to $\Omega_{I}$. The above considerations on the multiplicities of these functions entail that there exists an (intrinsic) holomorphic function $\beta_{I}$ on $\Omega_{I}$ such that 
$\alpha=\beta_{I}\nu(f_{v}^{s})$ on $\Omega_{I}$; we denote by $\beta$ the regular extension of $\beta_{I}$ to $\Omega$.
As both $\alpha$ and $\nu(f_{v}^{s})$ are slice preserving, then $\beta$ is slice preserving too and
$\alpha=\beta\nu(f_{v}^{s})$ on $\Omega$ by the Identity Principle.
Now write $\nu(f_{v}^{s})f_{v}=\beta \nu(f_{v}^{s})W$, as $\nu(f_{v}^{s})$
is a non-identically zero slice preserving function, we obtain that $f_{v}=\beta W$.
\end{proof}

Thanks to system~\eqref{system1} and Lemma~\ref{betaW} we obtain $\nu(W^{s}\beta^{2})\beta W=\alpha W$
that entails

\begin{lemma}\label{solution} 
Let $g\in\So^{1}(\Omega)$ be such that $g_{v}=\alpha W$ as in equation~\eqref{gvalphaW}.
If $f_{v}$ is a $*$-logarithm of $g$, then there exists $\beta\in\So_{\R}(\Omega)$,
such that $f_{v}=\beta W$ and $\beta$ satisfies 
\begin{equation}\label{system2}
\begin{cases}
\mu(W^{s}\beta^{2})=g_{0},\\
\nu(W^{s}\beta^{2})\beta=\alpha.
\end{cases}
\end{equation} 
Viceversa if $\beta$ is a solution of the previous system, then $f_{v}=\beta W$ is a 
$*$-logarithm of $g$.
\end{lemma}

\begin{proof}
Suppose $f_{v}$ is a $*$-logarithm of $g$, i.e. a solution of $\exp_{*}(f_{v})=g$. Then, Lemma~\ref{betaW} ensures that there exists $\beta\in\So_{\R}(\Omega)$
such that $f_{v}=\beta W$. Thus $\exp_{*}(f_{v})=g$ is equivalent to 
\begin{equation*}
\begin{cases}
\mu(W^{s}\beta^{2})=g_{0},\\
\nu(W^{s}\beta^{2})\beta W=\alpha W.
\end{cases}
\end{equation*} 
As $W$ is not identically zero and $\nu(W^{s}\beta^{2})\beta, \alpha\in\So_{\R}(\Omega)$, we can cancel $W$ and hence we obtain system~\eqref{system2}.

Vice versa suppose that $\beta$ is a solution of system~\eqref{system2} and set $f=f_{v}=\beta W$. Thus
\begin{equation*}
\exp_{*}(f_{v})=\mu(f_{v}^{s})+\nu(f_{v}^{s})f_{v}=\mu(W^{s}\beta^{2})+\nu(W^{s}\beta^{2})\beta W=g_{0}+\alpha W=g.
\end{equation*}
\end{proof}

This result allows us to prove that any function with $*$-logarithm
carries along  a whole family of functions with $*$-logarithm, thus generalizing Remark~\ref{conjugate}.

\begin{corollary}\label{equivalentsystem}
Let $g\in\So^{1}(\Omega)$ be such that $g_{v}=\alpha W$, where $\alpha\in\So_{\R}(\Omega)$ and $W$ has neither real nor spherical zeroes. If 
$g$ has a $*$-logarithm, for any $U\in\So(\Omega)$ such that $U_{0}\equiv 0$ and $U^{s}\equiv W^{s}$, the function
$\tilde g=g_{0}+\alpha U$ has a $*$-logarithm as well.
\end{corollary}

\begin{proof}
Let $f_{0}+f_{v}$ be a $*$-logarithm of $g$. By Proposition~\ref{propos1}, either $\exp_{*}(f_{v})=g$
or $\exp_{*}(f_{v})=-g$. In the first case, Lemma~\ref{solution} shows that there exists $\beta\in\So_{\R}(\Omega)$ such that $f_{v}=\beta W$
and $\beta$ satisfies~\eqref{system2}. A straightforward computation shows that
$\exp_{*}(\beta U)=\tilde{g}$.  
If $\exp_{*}(f_{v})=-g$, we apply the above reasoning to $-g$, obtaining
that $\exp_{*}(f_{0}+\beta U)=\tilde{g}$.
\end{proof}


Our first positive result on the solvability of equation~\eqref{exp} deals with the more manageable case in which $g_{v}$ 
has no non-real isolated zeroes, that is the function $W$ appearing in equation~\eqref{alphaW} is never vanishing.
The next theorem provides the existence of a $*$-logarithm for this class
of functions.

\begin{theorem}\label{gvhasnozeroes} Let $\Omega$ be slice contractible. Then any $g\in\So^{*}(\Omega)$ such that $g_{v}$ has no non-real isolated zeroes has a $*$-logarithm. 
\end{theorem}

\begin{proof}
By Remark~\ref{r1} we can limit ourselves to the case $g^{s}\equiv 1$.
As $W$ is never vanishing, then~\cite[Corollary 3.2]{A-dF} guarantees the existence of a square root $\tau\in\So_{\R}(\Omega)$
of $W^{s}$ and system~\eqref{system2} becomes 
\begin{equation*}
\begin{cases}
\mu(\tau^{2}\beta^{2})=g_{0},\\
\nu(\tau^{2}\beta^{2})\beta=\alpha,
\end{cases}
\end{equation*}
or, equivalently
\begin{equation}\label{system3}
\begin{cases}
\mu(\tau^{2}\beta^{2})=g_{0},\\
\nu(\tau^{2}\beta^{2})\beta\tau=\alpha\tau.
\end{cases}
\end{equation}
Using the relation between the power series of $\mu$ and $\cos$ and $\nu$ and $\sin$, last system can be written as
\begin{equation}\label{system4}
\begin{cases}
\cos(\tau\beta)=g_{0},\\
\sin(\tau\beta)=\alpha\tau.
\end{cases}
\end{equation}
Since $g_{0}^{2}+\alpha^{2}\tau^{2}=g_{0}^{2}+\alpha^{2}W^{s}=g_{0}^{2}+g_{v}^{s}=1$, by Proposition~\ref{cos-sin} there exists $\gamma\in\So_{\R}(\Omega)$ which solves
\begin{equation}\label{system5}
\begin{cases}
\cos(\gamma)=g_{0},\\
\sin(\gamma)=\alpha\tau.
\end{cases}
\end{equation}
Now, setting $\beta=\gamma/\tau$, 
where the second term is well defined since $\So_{\R}(\Omega)$ is abelian and
$\tau$ is never vanishing, we have that $\beta$ is a solution of system~\eqref{system2} and hence, thanks to
Lemma~\ref{solution},
$f=f_{v}=\beta W$ is a $*$-logarithm of $g$. 
%
%
\end{proof}

Theorem~\ref{gvhasnozeroes} allows us to give a first local existence result for the $*$-logarithm.

\begin{corollary}\label{corlocal}
Let $g\in\So(\Omega)$ and $q_{0}\in\Omega$ be such that $g_{|\SF_{q_{0}}}$
is never vanishing and $g_{v}$ has no non-real isolated zeros on $\SF_{q_{0}}$.
Then there exists a circular slice-contractible neighborhood $\Omega_{0}$ of $\SF_{q_{0}}$
such that $g_{|\Omega_{0}}$ admits a $*$-logarithm.
\end{corollary}

\begin{proof}
Theorem~\ref{gvhasnozeroes} guarantees the existence of a 
$*$-logarithm of $g$ on any circular neighborhood $\Omega_{0}$ of $q_{0}$ such that $\Omega_{0}$ is slice contractible, provided $g_{v}$ has no non-real isolated zeroes on $\Omega_{0}$. 
Such $\Omega_{0}$ exists because the sets of spheres where $g$ vanishes
and those where $g_{v}$ has a non-real isolated zero are discrete
and do not contain $\SF_{q_{0}}$.
\end{proof}

%

When $\Omega$ is a slice domain, the previous local existence result can be
improved to a suitable slice subdomain. 

\begin{proposition}\label{local1} 
Let $\Omega$ be a slice domain, $g\in\So^{*}(\Omega)$ and let $q_{0}\in\Omega\setminus\R$ such that the sphere $\SF_{q_{0}}$ does not contain any 
non-real isolated zero of $g_{v}$. Then
there exists a slice neighborhood $\Omega_{0}$ of $q_{0}$ which is slice contractible where $g$ has a $*$-logarithm, i.e., there exists $f\in\So(\Omega_{0})$
such that $\exp_{*}(f)=g|_{\Omega_{0}}$.
\end{proposition}

\begin{proof}
Again, Theorem~\ref{gvhasnozeroes} yields the proof provided
we construct $\Omega_{0}$ as in thesis of the statement.

Since $\Omega$ is slice and $\Omega\cap\C_{q_{0}}$ is connected by arcs, we
can find a piecewise linear path joining $x_{0}\in\Omega\cap \R$ with
$q_{0}$ which touches the real line at $x_{0}$ only,  and is such that an $\varepsilon$-neighborhood $\mathcal{U}$ of this path in $\C_{q_{0}}$ is contractible and contained in $\Omega_{I}$ (See Figure~\ref{fig5}). By replacing $\Omega$ with the circularization of this domain we
can suppose that $\Omega$ is slice and slice contractible.

\begin{figure}[ht]
\includegraphics[width=8cm]{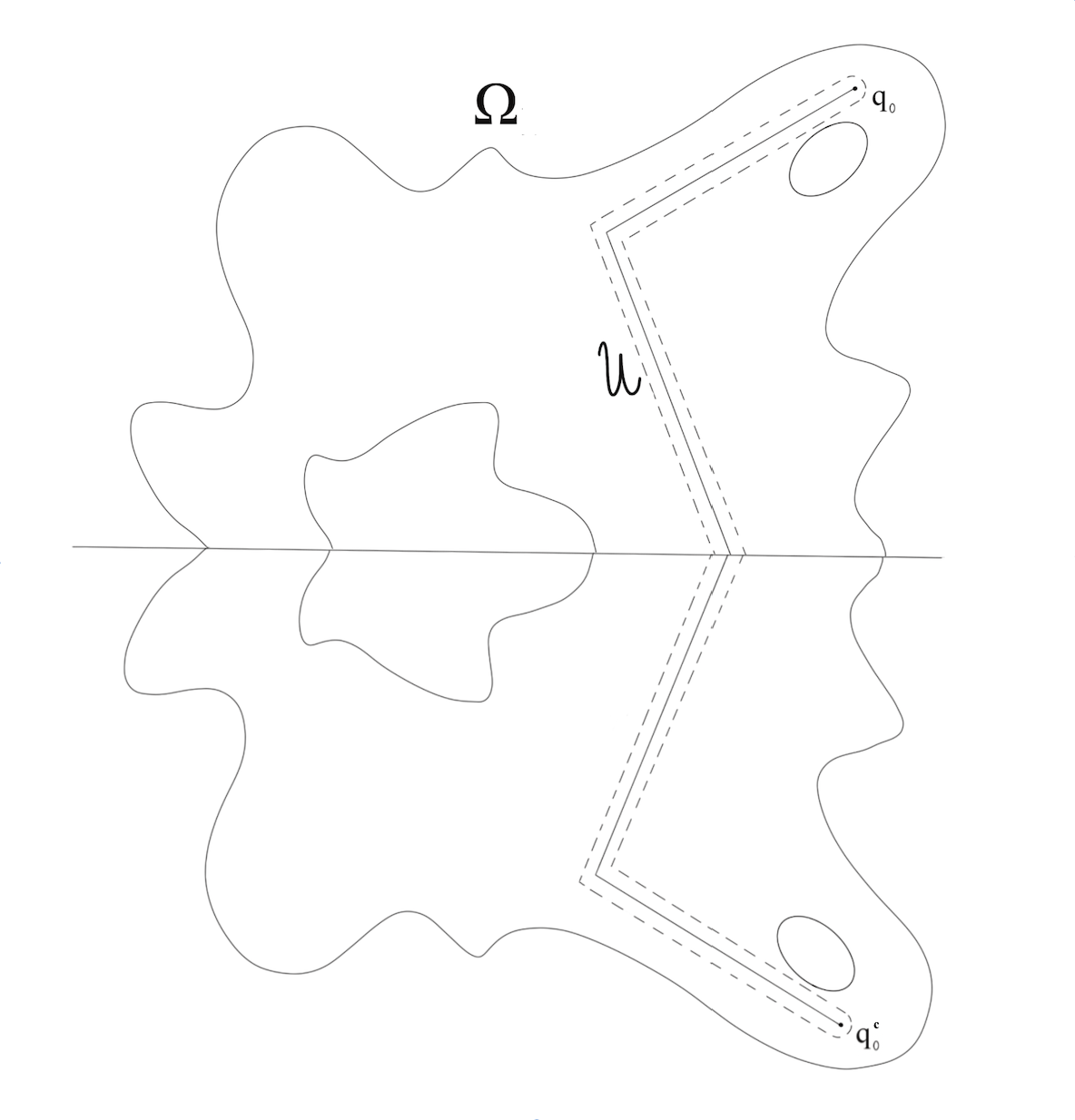}\label{fig5}
\caption{The domains involved in the first part of the proof of Proposition~\ref{local1}.}
\end{figure}

Thanks to~\cite[Corollary 3.7]{Gal-S} we can suppose that either $\Omega$ equals $\HH$ or the unitary ball $\mathbb{B}\subset \HH$ centered in $0$. 
Let $\Omega_{1}\subset\subset \Omega$ be a ball centered at the origin containing $q_{0}$. Lemma 3.11 in~\cite{G-S-St} entails that the set of non-real isolated zeroes of $g_{v}$
contained in $\Omega_{1}$ is finite. 
Let us denote by $S_{1},\dots S_{N}$ the spheres containing the non real
isolated zeroes of $g_{v}$.
Take a closed interval $\ell\subset\Omega_{1}\cap\R$ and consider the infinitely many segments joining
$q_{0}$ to the points of $\ell$. 
As $\mathcal{F}:=(S_{1}\cup\dots\cup S_{N})\cap\C_{q_{0}}$ is finite we can find a segment
$M\subset \Omega_{1}\cap\C_{q_{0}}$ joining $q_{0}$ to a point in $\ell$
which does not intersect $\mathcal{F}$. As $\mathcal{F}$ is symmetric with respect to conjugation in $\C_{q_{0}}$ and $M$ is compact, 
we can find a simply connected neighborhood $\mathcal{V}$ of $M$ in $\C_{q_{0}}$ symmetric
with respect to conjugation in $\C_{q_{0}}$ wich does not intersect $\mathcal{F}$.
Then, the circularization of $\mathcal{V}$ is the required $\Omega_{0}$.
\end{proof}


We now continue our investigation in search of a $*$-logarithm of $g$
in $\mathcal{R}^{1}(\Omega)$.
By Proposition~\ref{propos1}, up to a change of sign of $g$ if $\Omega$
is a product domain, we can limit ourselves to look for solutions of 
$\exp_{*}(f_{v})=g$, with the necessary condition that $g_{0}(q_{0})=1$
for any $q_{0}$ that is a non-real isolated zero of $g_{v}$.

Before stating the theorem we notice that for any $g\in\So^{*}(\Omega)$, the set $g_{0}^{-1}((-\infty,-1])$ is a circular set
because it is a union of pre-images of real points by the slice preserving function $g_{0}$.

\begin{theorem}\label{local2} 
Let $\Omega$ be slice contractible and $g\in\So^{1}(\Omega)$ be
such that for any $q_{0}\in\Omega$ that is a non-real isolated zero of $g_{v}$
we have $g_{0}(q_{0})=1$. Then,
on every connected component of $\Omega\setminus g_{0}^{-1}((-\infty,-1])$,
there exists a $*$-logarithm of $g$.
%
\end{theorem}

\begin{proof}
Let us denote by $\mathcal{U}$ a connected component of $\Omega\setminus g_{0}^{-1}((-\infty,-1])$. Notice that, as $g_{0}$ is slice preserving, then $\mathcal{U}$ is a circular domain. We claim that $\mathcal{U}$ is a domain where equation~\eqref{exp} admits a solution.
%
%
%
%

Let us write $g_{v}=\alpha W$ on $\Omega$ as in formula~\eqref{alphaW}.
Our choice of  $\mathcal{U}$ entails that $g_{0}(\mathcal{U})\subset \HH\setminus (-\infty,-1]$. 
Since the function $\varphi$ given in Definition~\ref{varphi} and $g_{0}$ are slice preserving,
then $\varphi\circ g_{0}:\mathcal{U}\to\mathcal{D}_{0}$ is a well defined slice preserving function. 
Thanks to Remark~\ref{zero}, the function $\nu\circ\varphi\circ g_{0}$ is a never vanishing slice preserving regular function on $\mathcal{U}$. Now set 
\begin{equation}\label{formulabeta}
\beta=\frac{\alpha}{\nu\circ\varphi\circ g_{0}}.
\end{equation}
We claim that $\beta$ is a solution of system~\eqref{system2} on $\mathcal{U}$.

First of all recall that $\mu\circ\varphi=\mbox{id}|_{\HH\setminus(-\infty,-1]}$. Thanks to this relation, the first equality in system~\eqref{system2} 
is satisfied if $\varphi\circ g_{0}=\beta^{2}W^{s}$. 
By squaring Equality~\eqref{formulabeta} we have 
\begin{equation}\label{formulabeta2}
\beta^{2}W^{s}=\frac{\alpha^{2}W^{s}}{(\nu\circ\varphi\circ g_{0})^{2}}.
\end{equation}
If $q\in\mathcal{U}$ is such that $(\varphi\circ g_{0})(q)=0$, then $g_{0}(q)=1$ (see Remark~\ref{zero});
as $g^{s}\equiv 1$ we then have $\alpha^{2}(q) W^{s}(q)=0$. Formula~\eqref{formulabeta2} implies 
$\beta^{2}(q)W^{s}(q)=0=(\varphi\circ g_{0})(q)$.
Suppose now that $q\in\mathcal{U}$ is such that $(\varphi\circ g_{0})(q)\neq0$. Then the following chain of equalities is
due to the fact that $g^{s}\equiv 1$, to Formula~\eqref{sommaquadrati} and to the fact that $\mu\circ\varphi=\mbox{id}|_{\HH\setminus(-\infty,-1]}$:
\begin{align*}
\beta^{2}(q)W^{s}(q)=\frac{\alpha^{2}(q)W^{s}(q)}{(\nu(\varphi(g_{0}(q)))^{2}}&=\frac{1-g_{0}^{2}(q)}{(\nu(\varphi(g_{0}(q))))^{2}}=\frac{(1-g_{0}^{2}(q))\cdot\varphi(g_{0}(q))}{(\nu(\varphi(g_{0}(q))))^{2}\cdot\varphi(g_{0}(q))}\\
&=\frac{(1-g_{0}^{2}(q))\cdot\varphi(g_{0}(q))}{1-(\mu(\varphi(g_{0}(q))))^{2}}=\frac{(1-g_{0}^{2}(q))\cdot\varphi(g_{0}(q))}{1-g_{0}^{2}(q)}=\varphi(g_{0}(q)).
\end{align*}
Now, since $\beta^{2}W^{s}=\varphi\circ g_{0}$, Equality~\eqref{formulabeta} immediately gives
$$
\nu(\beta^{2}W^{s})\cdot\beta=\alpha,
$$
which is the second equation of system~\eqref{system2}. Finally, 
thanks to Lemma~\ref{solution}, 
the assertion follows by setting $f=f_{v}=\beta\cdot W|_{\mathcal{U}}$.
\end{proof}

%

\begin{corollary}
Let $\Omega$ be a slice contractible product domain and $g\in\So^{1}(\Omega)$. Assume that for any $q_{0}\in\Omega$ that is a non-real isolated zero of $g_{v}$
we have that $g_{0}(q_{0})=-1$. Then,
on every connected component of $\Omega\setminus g_{0}^{-1}([1,+\infty))$,
there exists a $*$-logarithm of $g$.
 \end{corollary}
 
 \begin{proof}
 Set $\tilde g= \exp_{*}(\pi\mathcal{J})g=-g$. Then $\tilde{g}$ satisfies the hypotheses of
 Theorem~\ref{local2} and, therefore, there exists $f$ such that $\exp_{*}(f)=\tilde{g}$.
 A trivial computation gives $\exp_{*}(\pi\mathcal{J}+f)=g$.
 \end{proof}

It is worth observing that the difficulty of the proof of Proposition~\ref{local1} is of a purely topological nature, since the 
existence of a circular neighborhood of $q_{0}$ where $W^{s}$ is never vanishing is trivial, but the key point  
is that we are looking for a \textit{slice} circular neighborhood of $q_{0}$ whose
intersection with any slice is simply connected. On the contrary, the proof of Theorem~\ref{local2} has to overcome a problem
of analytical nature: indeed the existence of a circular neighborhood $\mathcal{U}$ of~$q_{0}$ 
such that $g_{0}((-\infty,-1])\cap\mathcal{U}=\emptyset$ immediately follows by the continuity of the function $g_{0}$,
while the construction of the function that gives the logarithm of $g$ on $\mathcal{U}$ requires the sharp analytical 
properties of the function $\mu$
obtained in Section~\ref{mu}.

%
%
%
%
%
In particular we are able to overcome this double kind of difficulties when
suitable topological hypothesis  allow us
to succeed in glueing three different solutions: one which is defined near
the non-real isolated zeroes of $g_{v}$ and two which are given on suitable slice-contractible domains  which do not contain the non-real isolated zero.

The idea of the proof is to solve (uniquely if the domain is slice) near the ``bad points'' (i.e., the non-real isolated zeroes of $g_{v}$) and to use this solution to select two suitable solutions in two appropriate (i.e., slice contractible),
domains whose union is exactly given by $\Omega$ minus the spheres containing
the bad points. In order this kind of reasoning works, the key problems shows two aspects. First of all,
the non-real isolated zeroes of $g_{v}$ could belong to different connected components of $\Omega\setminus g_{0}^{-1}((-\infty,-1])$ and thus 
we could not be sure that the leaves we selected around a point 
agree also around a different zero of $g_{v}$.
Secondly, even if all the non-real isolated zeroes of $g_{v}$ belong to
the same connected component $\mathcal{U}$ of $\Omega\setminus g_{0}^{-1}((-\infty,-1])$,
we have no information on the topology of $\mathcal{U}$ itself.
So the construction of the two slice simply connected domains whose union
is $\Omega$ minus the spheres where $g_{v}$ has non-real isolated zeroes, could give
a domain which does not allow to apply analytic continuation around each
of such zeroes.

The following statement describes a situation in which the existence of a $*$-logarithm holds. Recall the definition of $\mathbb{D}$ as $\mathbb{D}:=\{z\in\C_{i}\,|\,|z-2i|<1\}\times \SF$.

\begin{theorem}\label{final-result}
Let $\Omega$ be one among $\mathbb{B}$, $\HH$ or $\mathbb{D}$. Let $g\in\So^{1}(\Omega)$ be such that 
\begin{itemize}
\item $g_{v}$ has
a finite number of non-real isolated zeros $\{q_{1},\dots,q_{N}\}$;
\item $g_{0}(q_{\ell})=1$ for all $\ell=1,\dots, N$;
\item the union $\SF_{q_{1}}\cup\dots\cup\SF_{q_{N}}$ is contained in a unique connected
component $\mathcal{U}$ of $\Omega\setminus g_{0}^{-1}((-\infty,-1])$.
\end{itemize}
If for some $I\in\SF$ (and hence for any) the set $\mathcal{U}_{I}^{+}=\mathcal{U}\cap\C_{I}^{+}$ is convex and $\mathcal U$ is slice if $\Omega$ is, then there exists a slice regular $*$-logarithm of $g$. 
\end{theorem}

\begin{proof}

First of all choose any imaginary unit $I\in\SF$ and denote by $q_{\ell}'=\SF_{q_{\ell}}\cap\C_{I}^{+}$. 

Moreover choose $2N$ mutually disjoint outwarding segments (or rays in the case $\Omega=\HH$) $s_{\ell}, \sigma_{\ell},\subset \C^+_{I}\setminus \R$ starting from $q_{\ell}'$, for $\ell=1,\dots, N$ and 
such that $\Omega_{I}\setminus (s_{1}\cup \overline s_{1}\cup\dots\cup s_{N}\cup \overline s_{N})$ and $\Omega_{I}\setminus (\sigma_{1}\cup \overline \sigma_{1}\cup\dots\cup \sigma_{N}\cup \overline \sigma_{N})$ are contractible if $\Omega$ is slice and have two contractible connected components if $\Omega$ is product. We denote by $\widehat\Omega$ and
$\widetilde\Omega$ the circularizations of $\Omega_{I}\setminus (s_{1}\cup \overline s_{1}\cup\dots\cup s_{N}\cup \overline s_{N})$ and $\Omega_{I}\setminus (\sigma_{1}\cup \overline \sigma_{1}\cup\dots\cup \sigma_{N}\cup \overline \sigma_{N})$, respectively.
In other words we have the following equalities 
\begin{align*}
\widehat\Omega&=\Omega\setminus((s_{1}\cup\dots\cup s_{N})\times \SF),\\
\widetilde\Omega&=\Omega\setminus((\sigma_{1}\cup\dots\cup \sigma_{N})\times \SF).
\end{align*}
By Theorem~\ref{local2} we can find a $*$-logarithm $f_{\mathcal{U}}$ of $g$ on $\mathcal{U}$, while by Theorem~\ref{gvhasnozeroes} we can find $\widehat h\in\So(\widehat \Omega)$
and $\widetilde h\in\So(\widetilde\Omega)$ which are $*$-logarithms of $g$ on $\widehat \Omega$ and $\widetilde \Omega$, respectively.


As $\mathcal{U}_{I}^{+}$ is convex, then both
$\mathcal{U}_{I}^{+}\setminus (s_{1}\cup\dots\cup s_{N}))$ and $\mathcal{U}_{I}^{+}\setminus (\sigma_{1}\cup\dots\cup \sigma_{N})$ are connected; we will denote them by $\widehat{\mathcal U}_I^+$ and $\widetilde {\mathcal U}_I^+$ respectively.
Thus, also their circularizations $\mathcal U\cap \widehat \Omega$ and $\mathcal U\cap \widetilde \Omega$ 
are connected and will be denoted by $\widehat{\mathcal U}$ and $\widetilde{\mathcal U}$.
Moreover, $(\widehat\Omega\cap \widetilde\Omega)\cap \C_{I}^{+}=\Omega_{I}^{+}\setminus (s_{1}\cup\dots\cup s_{N}\cup\sigma_{1}\cup\dots\cup \sigma_{N})$ is the union of a $N+1$ connected components which are given by $N$ ``triangles'' $T_\ell$ with vertex in $q'_\ell$ and whose boundary in $\Omega_I^{+}$ is given by $s_\ell\cup \sigma_\ell$
and a connected component which is the complement of these triangles and will be denoted by $\Omega_I^0$.

\begin{figure}[ht]
\includegraphics[width=10cm]{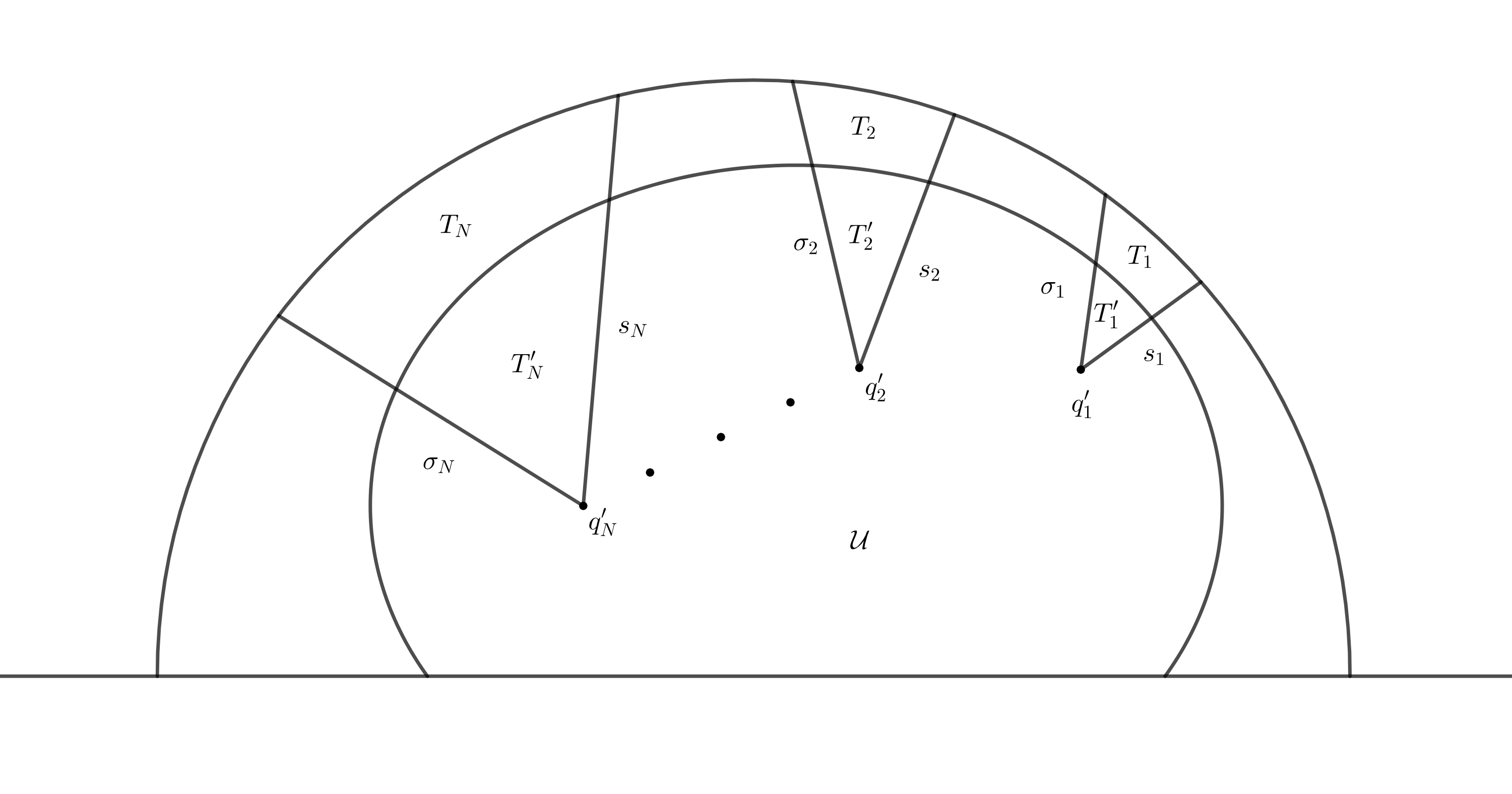}
\caption{An overview of the above geometric construction.}
\end{figure}

Again, the convexity of $\mathcal{U}_{I}^{+}$ gives that $(\mathcal U\cap \widehat\Omega\cap \widetilde\Omega)\cap \C_{I}^{+}=\mathcal U_{I}^{+}\setminus (s_{1}\cup\dots\cup s_{N}\cup\sigma_{1}\cup\dots\cup \sigma_{N})$ is
the union of a $N+1$ connected components which are given by $N$ smaller ``triangles'' $T'_\ell=T_\ell \cap \mathcal U_{I}^{+}$ with vertex in $q'_\ell$ and whose boundary in $\mathcal U_I^{+}$ is given by $(s_\ell\cup \sigma_\ell)\cap \mathcal U_{I}^{+}$ and a connected component which is the complement of these smaller triangles in $\mathcal U_{I}^{+}$ and will be denoted by $\mathcal U_I^0$.

The slice-contractibility of $\widehat\Omega$ and $\widetilde\Omega$ and the fact that $g_v$ has no non-real isolated zeroes in both these domains ensure that both $\hat h_v$ and $\tilde h_v$ have no non-real isolated zeroes and thus imply the existence of a square root $\sqrt{\hat h_{v}^{s}}$ on $\widehat\Omega$ and $\sqrt{\tilde h_{v}^{s}}$ on $\widetilde\Omega$. We also set $\widehat H_{v}=\frac{\hat h_{v}}{\sqrt{\hat h_{v}^{s}}}$ and 
$\widetilde H_{v}=\frac{\tilde h_{v}}{\sqrt{\tilde h_{v}^{s}}}$. 

We first perform the proof in the case when $\Omega$ is slice, by our assumptions also $\mathcal U$ is slice too.

Theorem~\ref{unicity} and Remark~\ref{remarkHV}, entail the existence of $\widehat m, \widetilde m\in\mathbb{Z}$
such that 
\begin{align}
f_\mathcal U-\hat h&=
2\widehat m\pi\hat H_{v} \qquad \text{on} \ \ \widehat{\mathcal U}
\label{prima-diff}\\
f_\mathcal U-\tilde h&=
2\widetilde m\pi \tilde H_{v} \qquad \text{on} \ \ \widetilde{\mathcal U} \label{seconda-diff}
\end{align}


Now set $\hat f=\hat h+2\pi \widehat m \widehat H_{v}$ on $\widehat \Omega$ and $\tilde f=\tilde h+2\pi \widetilde m \widetilde H_{v}$ on $\widetilde \Omega$. Equality~\eqref{prima-diff} entails $\hat f=f_{\mathcal{U}}$ on $\widehat{\mathcal{U}}$, while equality~\eqref{seconda-diff} gives $\tilde f=f_{\mathcal{U}}$ on $\widetilde{\mathcal{U}}$. Thus $\hat f=\tilde f$ on $\widehat{\mathcal{U}}\cap\widetilde{\mathcal{U}}$. As $\widehat{\mathcal{U}}\cap\widetilde{\mathcal{U}}$
contains accumulation points in any of the $N+1$ connected components of $\widehat{\Omega}\cap\widetilde{\Omega}$, then the Identity Principle implies $\hat f=\tilde f$ 
on $\widehat\Omega\cap\widetilde\Omega$.

Setting 
$$
f(q)=\begin{cases}
\widehat f(q),& \mbox{if }q\in\widehat \Omega,\\
\widetilde f(q),& \mbox{if }q\in\widetilde \Omega,\\
f_{\mathcal{U}}(q),& \mbox{if }q\in \mathcal{U}
\end{cases}
$$
gives a well defined slice regular function which is a $*$-logarithm of $g$ on $\Omega$.

We now turn our attention to the case in which $\Omega$ is product, which of course
entails that $\mathcal{U}$ is product, too. 

Theorem~\ref{unicity} entails the existence of $\widehat n,\widehat m, \widetilde n, \widetilde m\in\mathbb{Z}$ with $\widehat n\equiv  \widehat m$ $($mod. $2)$ and $\widetilde n\equiv  \widetilde m$ $($mod. $2)$, such that 
\begin{align}
f_\mathcal U-\hat h&=\pi\widehat n\mathcal{J}+\pi\widehat m\widehat H_{v} \qquad \text{on} \ \ \widehat{\mathcal U}
\label{prima-diff-nuovo}\\
f_\mathcal U-\tilde h&=\pi\widetilde n\mathcal{J}+\pi\widetilde m\widetilde H_{v} \qquad \text{on} \ \ \widetilde{\mathcal U}
 \label{seconda-diff-nuovo}
\end{align}
Setting again $\hat f= \hat h+\pi\widehat n\mathcal{J}+\pi\widehat m\widehat H_{v}$ on $\widehat \Omega$
and $\tilde f= \tilde h+\pi\widetilde n\mathcal{J}+\pi\widetilde m\widetilde H_{v}$ on $\widetilde \Omega$
and reasoning as above, gives the existence of a $*$-logarithm of $g$ on $\Omega$.
%
%
%
%
%
%
\end{proof}

In the case when $\Omega$ is product, the second condition of the previous theorem can be relaxed.

\begin{corollary}
Let $g\in\So^{1}(\mathbb{D})$ be such that 
\begin{itemize}
\item $g_{v}$ has
a finite number of non-real isolated zeros $\{q_{1},\dots,q_{N}\}$;
\item $g_{0}(q_{\ell})=-1$ for all $\ell=1,\dots, N$;
\item the union $\SF_{q_{1}}\cup\dots\cup\SF_{q_{N}}$ is contained in a unique connected
component $\mathcal{U}$ of $\Omega\setminus g_{0}^{-1}((-\infty,-1])$.
\end{itemize}
If for some $I\in\SF$ (and hence for any) the set $\mathcal{U}_{I}^{+}=\mathcal{U}\cap\C_{I}^{+}$ is convex, then there exists a slice regular $*$-logarithm of $g$. 
\end{corollary}

\begin{proof}
By applying Theorem~\ref{final-result} to $-g$, we find a $*$-logarithm $f$ of $-g$,
then the function $f+\pi\mathcal{J}$ is a $*$-logarithm of~$g$.
\end{proof}

\begin{remark}
We notice that the statement of Theorem~\ref{final-result} can be generalized to a larger variety of
domains and functions. Indeed, the techniques we use in the proof can be applied 
 when $\Omega$ is slice contractible, all the non-real isolated zeroes of $g_{v}$
belong to the same connected component $\mathcal{U}$ of $\Omega\setminus g_{0}^{-1}((-\infty,-1])$ and for any non-real isolated zero we can ``draw'' two paths issuing from the non-real isolated zeroes of $g_{v}$ which give two contractible subdomains of $\Omega_{I}$ 
and do not disconnect $\mathcal{U}_{I}$.
\end{remark}

\bibliographystyle{amsplain}

\begin{thebibliography}{Ci-Kr}


\bibitem{A-CVEE}
A. Altavilla, Some properties for quaternionic slice-regular functions on domains without real points. 
Complex Var. Elliptic Equ. 60, n. 1 (2015), 59--77.



\bibitem{A-dF} A. Altavilla, C. de Fabritiis, $*$-exponential of slice-regular functions, Proc. Amer. Math. Soc. 147, (2019), 1173--1188.

\bibitem{A-dFAMPA} A. Altavilla, C. de Fabritiis, s-Regular functions which preserve a complex slice, Ann. Mat. Pura Appl. (4) 197:4, (2018), 1269--1294.

\bibitem{A-dFSylvester} A. Altavilla, C. de Fabritiis, Equivalence of slice semi-regular functions via Sylvester operators, Linear Algebra Appl., 607C (2020), 151--189.


\bibitem{A-dFConcrete} A. Altavilla, C. de Fabritiis, 
Applications of the Sylvester operator in the space of slice semi-regular functions, Concrete Operators, 7(1) (2020), 1--12.
%


%
%
\bibitem{C-G-S-St}  F. Colombo, G. Gentili, I. Sabadini, D. C. Struppa, Extension results for slice regular functions
of a quaternionic variable. Adv. Math. 222(5), (2009), 1793--1808.
\bibitem{C-GC-S} F. Colombo, J. Oscar Gonzalez-Cervantes, I. Sabadini, The C-property for slice regular functions and applications to the Bergman space, Compl. Var. Ell. Eq., 58, n. 10 (2013), 1355--1372.
\bibitem{C-S-St-2}  F. Colombo, I. Sabadini, D. C. Struppa, Entire Slice Regular Functions, SpringerBriefs in Mathematics, Springer, 2016.
\bibitem{Gal-OGC-Sabadini}  S. G. Gal, J. Oscar Gonzalez-Cervantes, I. Sabadini, 
On Some Geometric Properties of Slice Regular Functions of a quaternion variable, Compl. Var. Ell. Eq., 60, n. 10 (2015), 1431--1455.

\bibitem{Gal-S} S. G. Gal, I. Sabadini, Approximation by polynomials on quaternionic compact sets, Math. Meth. Appl. Sci., 38 (2015), 3063--3074. 



\bibitem{GPV} G. Gentili, J. Prezelij, F. Vlacci,
Slice conformality: Riemann manifolds and logarithm on quaternions and octonions. In preparation.


\bibitem{G-S-St} G. Gentili, C. Stoppato, D. C. Struppa, Regular Functions of a Quaternionic Variable,  Springer Monographs in Matehmatics, Springer, 2013. 
\bibitem{G-M-P} R. Ghiloni, V. Moretti, A. Perotti, Continuous Slice Functional Calculus in Quaternionic Hilbert Spaces,
Rev. Math. Phys. 25 (2013), 1350006-1--1350006-83.
\bibitem{G-P} R. Ghiloni, A. Perotti, Slice regular functions on real alternative algebras, Adv. in Math., v. 226, n. 2 (2011),  1662--1691.



\bibitem{GPSalgebra} R. Ghiloni, A. Perotti, C. Stoppato, The algebra of slice functions, Trans. of Amer. Math. Soc., Volume 369, N.7, (2017), pp.4725--4762
\bibitem{GPSadvances} R. Ghiloni, A. Perotti, and C. Stoppato. Singularities of slice regular functions over real
alternative $*$-algebras. Adv. Math., (2017), 305:1085--1130.

\bibitem{GPSdivision} R. Ghiloni, A. Perotti, C. Stoppato, Division algebras of slice functions. Proceedings of the Royal Society of Edinburgh: Section A Mathematics, 150(4), (2020), 2055--2082. 










\bibitem{R-W} G. Ren, X. Wang, Slice Regular Composition Operators, Compl. Var. Ell. Equ., 61-5 (2015), 682--711.

\bibitem{V} F. Vlacci, Regular Composition for Slice--Regular Functions of Quaternionic Variable,
Advances in Hypercomplex Analysis, Springer INDAM Series 1, Springer-Verlag 141--148, 2013.





\end{thebibliography}

\end{document}